\documentclass[10pt]{article}

\usepackage{amsmath,amsfonts,amssymb,amsthm,mathrsfs}
\usepackage{graphicx}
\usepackage{stmaryrd}
\usepackage{hyperref}
\usepackage{tikz}

\usepackage[left=3.5cm,right=3.5cm,top=3.5cm,bottom=3.5cm]{geometry}

\theoremstyle{definition}
\newtheorem{definition}{Definition}
\newtheorem{remark}[definition]{Remark}
\newtheorem{example}[definition]{Example}

\theoremstyle{mytheorem}
\newtheorem{theorem}[definition]{Theorem}
\newtheorem{lemma}[definition]{Lemma}
\newtheorem{proposition}[definition]{Proposition}

\newtheorem{assumption}[definition]{Assumption}

\usepackage{color}

\renewcommand{\P}{\mathbf{P}} 
\newcommand{\E}{\mathbf{E}}

 \newcommand{\W}{\mathscr{W}}

\renewcommand{\S}{\mathcal{S}}
\newcommand{\Per}{{\rm Per}}

\newcommand{\Vol}{{\rm Vol}}

 \newcommand{\R}{\mathcal{R}}
\newcommand{\corners}{{\rm corners}}

\renewcommand{\u}{\mathbf{u}}

\newcommand{\A}{\mathcal{A}}
\renewcommand{\root}[1]{\underline{#1}}
 
\newcommand{\N}{\mathsf{N}}

\newcommand{\p}{\mathsf{P}}

\newcommand{\n}{\mathbf{n}}

\newcommand{\C} {\mathcal{C} }

\newcommand{\EC}{Euler characteristic}

\newtheorem{properties} {Properties}

\title {Bicovariograms and Euler characteristic of regular sets}

\author{Rapha\"el  Lachi\`eze-Rey \footnote{  {\sf raphael.lachieze-rey@parisdescartes.fr, Phone: +33\,1\,83\,94\,58\,45
  {Universit\'e Paris Descartes, Sorbonne Paris Cit\'e, MAP5, 45 Rue des Saints-P\`eres, 75006 Paris}}}}

\begin{document}
\date{}
\maketitle


\begin{abstract}

We establish an expression of the \EC~of a $r$-regular planar set in function of some variographic quantities. The usual $\mathcal{C} ^{2}$ framework is relaxed to a $\mathcal{C} ^{1,1}$ regularity assumption, generalising existing local formulas for the \EC. We give also general bounds on the number of connected components of a measurable set of $\mathbb{R}^{2}$ in terms of local quantities. These results are then combined to yield a new expression of the mean \EC~of a random regular set, depending solely on the third order marginals for arbitrarily close arguments. We derive results for level sets of some   moving average processes  and for the boolean model with non-connected polyrectangular  grains in $\mathbb{R}^{2}$. Applications to excursions of smooth bivariate random fields are derived in the companion paper \cite{LacEC2}, and applied for instance to $\C^{1,1}$ Gaussian  fields, generalising standard results.

\end{abstract}

 \textbf{keywords:} {Euler characteristic,   intrinsic volumes,   shot noise processes, boolean model\\}


\section*{Introduction}

Physicists and biologists are always in search of numerical indicators reflecting the microscopic and macroscopic behaviour of  tissue, foams, fluids, or  other spatial structures. The Euler characteristic, also called Euler-Poincar\'e characteristic,  is a favoured topological index  because its additivity properties   make it more manageable than connectivity indexes or Betti numbers. It is defined on a set $A\subseteq \mathbb{R}^{2}$ by
\begin{align}
\label{def:euler-cc}
\chi (A)=\#\{\text{bounded components of $A$}\}-\#\{\text{bounded components of $A^{c}$}\}.
\end{align} It is more generally an indicator of the regularity of the set, as an irregular structure is more likely to be shredded in many small pieces, or pierced by many holes, which results in a large value for $ | \chi (A) | $.

As an integer-valued  quantity, the \EC~can be easily measured and used in  estimation and modelisation procedures. It is an important indicator of the porosity of a random media \cite{AMMS, SWRS,Hil02}, it is used in brain imagery \cite{KilFri,TayWor}, astronomy, \cite{Mel90,Schmalzing,Mar15}, and many other disciplines. 
 See also \cite{ABBSW} for a general review of applied algebraic topology. 
In the study of  parametric random media or graphs, a small value of $ | \E \chi (A) | $ indicates the proximity of  the percolation threshold, when that makes sense. See   \cite{Okun}, or \cite{BDJR} in the discrete setting.

The mathematical  additivity property is expressed, for suitable sets $A$ and $B$, by the formula $\chi (A\cup B)=\chi (A)+\chi (B)-\chi (A\cap B)$, which applies recursively to finite such unions. 
In the validity domain of this formula,  the \EC~of a set can therefore be computed by summing local contributions.   The Gauss-Bonnet theorem formalises this notion for $\mathcal{C}^{2}$ manifolds, stating that the \EC~of a smooth set is the integral along the boundary of its Gaussian curvature. Exploiting the local nature of the curvature in applications seems to be a geometric challenge, in the sense that it is not always clear how to express the mean \EC~of a random set $F$ under the form 
\begin{align}
\label{eq:intro-local-EC}
\E \chi (F)=\lim_{\varepsilon \to 0}\int_{\mathbb{R}^{2}}\E\varphi_{\varepsilon } (x,F)dx
\end{align} where $\varphi_{\varepsilon } (x,F)$ only depends on  $ B(x,\varepsilon )\cap F$, where $B(x,\varepsilon )$ is the ball with center $x$ and radius $\varepsilon $. We propose in this paper a new formula of the form above, based on variographic tools, and valid beyond the $\mathcal{C}^{2}$ realm, and then apply it in a random setting. This paper is completely oriented towards probabilistic applications, it is not clear wether our formula has important implications in a purely geometric framework.

\section*{Approach}

In stochastic geometry and stereology, an important body of literature  is concerned with providing formulas for computing the \EC~of random sets, see for instance \cite{HLW,SchWei,KSS,NOP} and references therein.  Defined to be $1$ for every  convex  body, it is extended by additivity as
\begin{align*}
\chi (\cup _{i}C_{i})=-\sum_{I\subseteq [m],I\neq \emptyset }(-1)^{\# I}\mathbf{1}_{\{\cap _{i\in I}C_{i}\neq \emptyset \}}
\end{align*}  for finite unions of such sets. Even though this formula seems highly non-local, it is possible to express it as a sum over local contributions  using the Steiner formula, see (2.3) in \cite{KSS}, but it is difficult to apply it under this form.
 There has also been an intensive research around the \EC~of random fields excursions \cite{AdlSam,AufBen13,Mar15,EstLeo14,AzaWsc,TayWor}, based upon the works of Adler, Taylor, Sammorodnitsky, Worsley, and their co-authors, see the central monograph \cite{AdlTay07}. 
 We discuss in the companion paper \cite{LacEC2} the application of the present results  to level sets of random fields.

In this work, we give a relation between the \EC~of a bounded subset $F$ of $\mathbb{R}^{2}$ and some variographic quantities related to $F$.
Given any two orthogonal unit vectors $\u_{1},\u_{2}$, for $\varepsilon $ sufficiently small,
\begin{align}
\label{eq:intro-marginals}
\chi (F)&=\varepsilon ^{-2}\Big[ \Vol(F\cap (F-\varepsilon \u_{1})^{c}\cap (F-\varepsilon \u_{2})^{c})\\
\notag&\hspace{4cm}-
\Vol(F^{c}\cap (F+\varepsilon \u_{1})\cap (F+\varepsilon \u_{2})) \Big],\end{align}
where $\Vol$ is the $2$-dimensional Lebesgue measure.
This formula is valid under the assumption that $F$ is  {$\C^{1,1}$}, i.e. that  $\partial F$ is  a $\mathcal{C}^{1}$ submanifold of $\mathbb{R}^{2}$ with Lipschitz normal and finitely many connected components. See Example \ref{ex:disc} for the application of this formula to the unit disc.

In the context of a random closed set $F$, call $R_{\varepsilon }$ the right-hand member of \eqref{eq:intro-marginals}. If $\E\sup_{0\leqslant \varepsilon \leqslant 1 }R_{\varepsilon }$ is finite, the value of $\E\chi (F)$ can  be obtained as $\lim_{\varepsilon \to 0}\E R_{\varepsilon }$. The main asset of this formulation regarding classical approaches is that, to compute the mean \EC, one only needs to know the third-order marginal of $F$, i.e. the value of 
\begin{align*}
(x,y,z)\mapsto \P(x,y,z\in F),
\end{align*}
for $x,y,z$ arbitrarily close.
We also give   similar results for the intersection $F\cap W$, where $F$ is a random regular closed set and $W$ is a rectangular (or poly-rectangular) observation window. This step is necessary to apply the results to  a stationary set sampled  on a bounded portion of the plane.

In the present paper, we apply the principles underlying these formulas to obtain the mean \EC~for level sets of moving averages, also called shot noise processes, where the kernels are the indicator functions of random sets which geometry is adapted to the lattice approximation. Even though the geometry of moving averages level sets attracted interest in the recent literature \cite{AST,BieDes}, no such result seemed to exist in the literature \footnote{A  more general result has now been derived by Bierm\'e and Desolneux \cite{BieDesEuler}}. As a by-product, the mean \EC~of the associated boolean model is also obtained.

 These formulas are successfully applied to excursions of smooth random fields in the companion paper \cite{LacEC2}. For instance, in the context of Gaussian fields excursions, one can pass \eqref{eq:intro-marginals} to expectations under the requirement that the underlying field is $\mathcal{C}^{1,1}$, i.e. in the context of bivariate functions $\mathcal{C}^{1}$ with Lipschitz derivatives, plus additional moment conditions. This improves upon the classical theory \cite{AdlTay07} where fields have to be of  class $\mathcal{C}^{2}$ and satisfy a.s. Morse hypotheses. Here again, the resulting formulas only require the knowledge of the field's third order marginals  for arbitrarily close arguments.\\

\section*{Discussion}
 
 Equality \eqref{eq:intro-marginals}  gives in fact a direct relation between the \EC, also known as the Minkowski functional of order $0$, and the function $(x,y)\mapsto \Vol(F\cap (F+x)\cap (F+y))$. We call the latter function \emph{bicovariogram} of $F$, or variogram of order $2$, in reference to the \emph{covariogram} of $F$, defined by $x\mapsto \Vol(F\cap (F+x))$ (see \cite{Lan02,Galerne} or \cite{Serra} for more on covariograms). 
Let $\sigma $ be the normalized Haar measure on the $1$-dimensional circle $\mathcal{S} ^{1}$. The formula
\begin{align*}
 \Per(F)&=  \lim_{\varepsilon \to 0}\int_{\S^{1}}\varepsilon ^{-1}\left( \Vol(F)-\Vol(F\cap( F+\varepsilon u)) \right)\sigma (d\u),
\end{align*}developped in the context of random sets by Galerne \cite{Galerne}, and originating from the theory of functions of bounded variations \cite{AFP},  gives a direct relation between the first order variogram, and the perimeter of a measurable set $F$, which is also the Minkowski functional of order $1$ in the vocabulary of convex geometry. Completing the picture with the fact that $\Vol(F)$ is at the same time the second-order Minkowski functional and the variogram of order $0$, it seems that covariograms and Minkowski functionals are intrinsically linked. This unveils a new field of exploration, and raises the questions of extension to higher dimensions, with higher order variograms, and all Minkowski functionals. 

The present work  is limited to the dimension $2$ because, before engaging in a general theory, one must check that, at least in a particular case of interest, existing  results are improved. In the present work and the companion paper, the results are oriented towards  excursions of bivariate Gaussian fields, as they are of high interest in the literature. Despite the technicalities and difficulties, coming mainly from - 1 - the expression of topological estimates in terms of the regularity of the field,  and - 2 - dealing with boundary effects,  obtaining a formula valid for any random model, and relaxing the usual $\mathcal{C} ^{2}$ hypotheses to $\mathcal{C} ^{1}$ assumptions, provides a sufficiently strong motivation for pushing the theory further. Also, developing methods of proof and upper bounds in dimension $2$ will help developing them in more abstract  spaces.

Another motivation of the present work is that the amount of information   that can be retrieved from the variogram of a set is a central topic in the field of stereology, see for instance the recent work \cite{AveBia09} completing the confirmation of Matheron's conjecture. Through relation \eqref{eq:intro-marginals}, the data of the bicovariogram function with  arguments arbitrarily close to $0$ is sufficient to derive its Euler characteristic, and once again the extension to higher dimensions is a natural interrogation. 
 
  \subsection*{Plan}

The paper is organised as follows. We give in Section \ref{section:deterministic} some tools of image analysis, and the framework for stating our main result, Theorem \ref{th:as-conv}, which proves in particular \eqref{eq:intro-marginals}. These results are used to derive the mean \EC~of shot noise level sets and boolean model with polyrectangular grain. We then provide in Theorem \ref{th:bound-pixel-comp} a uniform bound for the number of connected components of a digitalised set, useful for applying Lebesgue's Theorem. In Section \ref{sec:random sets}, we introduce random closed sets and the conditions under which the previous results give a convenient expression for the mean \EC. Theorem \ref{prop:stationary-set} states hypotheses and results for homogeneous random models.

\section{\EC~of regular sets}
\label{section:deterministic}

Given a measurable set $A$ of $\mathbb{R}^{2}$, denote by $\Gamma  (A)$ the family of bounded connected components of $A$, i.e. the bounded equivalence classes of $A$ under the relation ``$x\in A$ is connected to $y\in A$ if there is a continuous path $\gamma: [0,1]\to A $ such that $\gamma (0)=x,\gamma (1)=y$''. We use the notation $\tilde \gamma   :=\gamma ([0,1])$ to indicate the image of such a path. Call $\mathcal{A}$ the class of sets of $\mathbb{R}^{2}$ such that $\Gamma (A)$ and $\Gamma (A^{c})$ are finite.  We call the sets of $\mathcal{A}$ the \emph{admissible sets of $\mathbb{R}^{2}$}, and  define for $A\in \mathcal{A}$, $$\chi (A)=\#\Gamma (A)-\#\Gamma (A^{c}).$$

The present paper is restricted to the dimension $2$, we therefore will not go further in the algebraic topology and homology theory underlying the definition of the \EC. 
The aim of this section is to provide a lattice approximation $A^{\varepsilon }$ of $A$   for which $\chi (A^{\varepsilon })$ has a tractable expression, and explore under what hypotheses on $A$ we have $\chi (A^{\varepsilon })\to \chi (A)$ as $\varepsilon \to 0$.

\paragraph{Some notation} For $x\in \mathbb{R}^{2}$, call $x_{[1]},x_{[2]}$ its components in the canonical basis. 
Also denote, for $x,y\in \mathbb{R}^{d},d\geqslant 1$, by $[x,y]$ the segment delimited by $x,y$, and $(x,y)=[x,y]\setminus \{x,y\}$.

\subsection{\EC~and image analysis}
\label{sec:image}

Practitioners compute the \EC~of a set $F\subset \mathbb{R}^{d}$ from a digital lattice approximation $F^{\varepsilon }$, where $\varepsilon $ is close to $0$. The computation of $\chi (F^{\varepsilon })$ is based on a linear filtering with a patch containing $2^{d}$ pixels, see \cite{NOP,Sva14,Sva15,Kid06}. 
Determining wether $\chi (F^{\varepsilon })\approx \chi (F)$ is a problem with a long history in image analysis and stochastic geometry.

 For $\varepsilon >0$, call $Z_{\varepsilon }=\varepsilon \mathbb{Z} ^{2}$ the square lattice with mesh $\varepsilon $, and say that two points of $ Z_\varepsilon $ are neighbours if they are at distance $\varepsilon $ (with the additional convention that a point is its own neighbour). Say that two points are connected if there is a finite path of connected points between them. If the context is ambiguous, we  use the terms \emph{grid-neighbour,grid-connected,} to not mistake it with the general $\mathbb{R}^{2}$ connectivity.  Call $\Gamma ^{\varepsilon }(M)$ the class of finite (grid-)connected components of a set $M\subseteq  Z_\varepsilon $.  
  We define in analogy with the continuous case, for $M \subseteq Z_{\varepsilon }$ bounded such that $\Gamma^{\varepsilon } (M),\Gamma ^{\varepsilon }(M^{c})$ are finite, 
\begin{align*}
\chi ^{\varepsilon }(M)=\#\Gamma ^{\varepsilon }(M)-\#\Gamma ^{\varepsilon }(M^{c}),
\end{align*} 
where $M^{c}=Z_{\varepsilon }\setminus M$. Remark in particular that two connected components touching exclusively through a corner are not grid-connected. 

Call $\u=\{\u_{1},\u_{2}\}$ the two canonical unit vectors or $\mathbb{R}^{2}$,
and define  for $A\subseteq \mathbb{R} ^{2},x\in \mathbb{R}^{2},\varepsilon >0$,
\begin{align*}
\Phi ^{\varepsilon } _{\text{out}}(x;A)&=\mathbf{1}_{\{x\in A,x+\varepsilon \u_{1}\notin A,x+\varepsilon \u_{2}\notin A\}},\\
\Phi ^{\varepsilon }_{\text{in}}(x;A)&=\mathbf{1}_{\{x\notin A,x-\varepsilon \u_{1}\in A,x-\varepsilon \u_{2}\in A\}},\\
\Phi _{X}^{\varepsilon }(x;A)&=\mathbf{1}_{\{x\in A;x+\varepsilon \u_{1}\notin A,x+\varepsilon \u_{2}\notin A,x+\varepsilon (\u_{1}+\u_{2})\in A\}}.
\end{align*}
Seeing also these functionals as discrete measures, define  
\begin{align*}
\Phi _{\text{out}}^{\varepsilon }(A)&=\sum_{x\in Z_{\varepsilon }}\Phi _{\text{out}}^{\varepsilon }(x;A),\\
\Phi _{\text{in}}^{\varepsilon }(A)&=\sum_{x\in Z_{\varepsilon  }}\Phi _{\text{in}}^{\varepsilon }(x;A),\\
\Phi _{X}^{\varepsilon }(A)&=\sum_{x\in Z_{\varepsilon }}\Phi _{X}^{\varepsilon }(x;A).
\end{align*}

The subscripts \emph{in} and \emph{out} refer to the fact  that $\Phi ^{\varepsilon } _{\text{out}}(A)$ counts the number of vertices of $A\cap Z_{\varepsilon }$ pointing outwards towards North-East, and $\Phi  ^{\varepsilon }_{\text{in}}(A) $ is the number of vertices pointing inwards towards South-West. 
Define for $A\subseteq \mathbb{R}^{2}$
 $$\chi ^{\varepsilon }(x;A)=\Phi _{\text{out}}^{\varepsilon }(x;A)-\Phi _{\text{in}}^{\varepsilon }(x;A).$$

The functional $\Phi _{X}^{\varepsilon }(A)$ is intended to count the number of \emph{$X$-configurations}. Such configurations are a nuisance for obtaining the \EC~by summing local contributions. Call $\mathcal{A}(Z_{\varepsilon })$ the class of bounded $M\subseteq Z_{\varepsilon }$ such that   $\Phi _{X}^{\varepsilon }(M)=\Phi _{X}^{\varepsilon }(M^{c})=0$, with $M^{c}=Z_{\varepsilon }\setminus M$.

\begin{lemma}
\label{lm:chi-eps-phi-0}
For $M\in \mathcal{A}(Z_{\varepsilon })$,
\begin{align}
\label{eq:discrete-EC}
\chi ^{\varepsilon }(M)=\sum_{x\in Z_{\varepsilon }}\chi ^{\varepsilon }(x;M).
\end{align}
\end{lemma}
{
 \begin{proof}
It is well known that, viewing $M$ as a subgraph of $\mathbb{Z} ^{2}$, the \EC~of $M$ can be computed as $\chi (M)=V-E+F$ where $V$ is the number of vertices of $M$, $E$ is its number of edges, and $F$ is the number of facets, i.e. of points $x\in M$ such that $x+\varepsilon \u_{1},x+\varepsilon \u_{2},x+\varepsilon (\u_{1}+\u_{2})\in M$. We therefore have 
\begin{align*}
\chi ^{\varepsilon }(M)=\sum_{x\in \mathbb{Z} ^{2}}\left[
\mathbf{1}_{\{x\in M\}}-\sum_{i=1}^{2}\mathbf{1}_{\{\{ x,x+\varepsilon \u_{i}\}\text{\rm{ is an edge of $M$}}\}}+\mathbf{1}_{\{x\text{\rm{ is the bottom left corner of a facet}}\}}
\right].
\end{align*}
For each $x$, the summand is in $\{-1,0,1\}$ and can be computed in function of the configuration\\
 $(\mathbf{1}_{\{x\in M\}},\mathbf{1}_{\{x+\varepsilon \u_{1}\in M\}},\mathbf{1}_{\{x+\varepsilon \u_{2}\in M\}} ,\mathbf{1}_{\{x+\varepsilon (\u_{1}+\u_{2})\in M\}})\in \{0,1\}^{4}$. Enumerating all the possible configurations and  noting that the configurations $(1,0,0,1)$ and $(0,1,1,0)$ do not occur due to the assumption $M\in \mathcal{A}(Z_{\varepsilon })$, it yields that indeed only the configurations $(1,0,0,0)$ give  $+1$ and only the configurations $(1,1,1,0)$ give  $-1$, which gives the conclusion. 
\end{proof}
 
This formula amounts to  a linear filtering of the set by a $2\times 2$ discrete patch, and is already known and used in image analysis and in physics on discrete images. 
Analogues of this formula \cite{AKPM,NOP}  exist also  in higher dimensional grids, but the dimension $2$ seems to be the only one where an anisotropic form is valid, see \cite{Sva15} for a discussion on this topic. The anisotropy is not indispensable to the  results  discussed in this paper, but gives more generality and simplifies certain formulas. An  isotropic formula can be obtained by averaging over the $4$ directions.}

Given a  subset $A$ of $\mathbb{R}^{2}$ and $\varepsilon >0$, we are interested here in the topological properties of the \emph{Gauss digitalisation of $A$}, defined  by $  [A]^{\varepsilon }:=A\cap Z_{\varepsilon }$. Define $\p=\varepsilon [-\frac{1}{2},\frac{1}{2})^{2}$, and 
\begin{align*}
A^{\varepsilon }=\bigcup _{x\in [A]^{\varepsilon }}(x+\p)
\end{align*}is the \emph{Gauss reconstruction } of $A$ based on $Z_{\varepsilon }$. In some unambiguous cases, the notation is simplified to $[A]=[A]^{\varepsilon }$.
In this paper, we refer to a \emph{pixel} as a set $\p+x$, for $x\in \mathbb{R}^{2}$ not necessarily in $Z_{\varepsilon }$. 

\paragraph{Notation}
We also use the notation, for $x,y\in Z_{\varepsilon }$, $\llbracket x,y \rrbracket=[x,y]\cap  Z_{\varepsilon }, \llparenthesis x,y\rrparenthesis=\llbracket x,y\rrbracket\setminus\{x,y\}$.

\begin{properties}
\label{ppt:gauss}
For $A\subseteq \mathbb{R}^{2}$, $A^{\varepsilon }$ is connected in $\mathbb{R}^{2}$ if $[A]^{\varepsilon }$ is grid-connected in $\mathbb{Z} ^{2}$. The converse might not be true because of pixels touching through a corner, but this subtlety does not play any role in this paper, because sets with $X$-configurations are systematically discarded. We also have $\chi ^{\varepsilon }([A])=\chi (A^{\varepsilon })$ if $[A]\in \mathcal{A}(Z_{\varepsilon })$ because connected components of $A$ (resp. $A^{c}$) can be uniquely associated to grid-connected components of $[A]$ (resp. $[A^{c}]$).

Most set operations commute with the operators $(.)^{\varepsilon },[\cdot ]^{\varepsilon }$. For any $A,B\subseteq \mathbb{R}^{2}$, $[A\cup B]^{\varepsilon }=[A]^{\varepsilon }\cup [B]^{\varepsilon },[A\cap B]^{\varepsilon }=[A]^{\varepsilon }\cap [B]^{\varepsilon },[\mathbb{R}^{2}\setminus A]^{\varepsilon }=Z_{\varepsilon }\setminus [A]^{\varepsilon },$ and those properties are followed by the reconstructions $(A\cup B)^{\varepsilon }=A^{\varepsilon }\cup B^{\varepsilon }, (A\cap B)^{\varepsilon }=A^{\varepsilon }\cap B^{\varepsilon },(A\setminus B)^{\varepsilon }=A^{\varepsilon }\setminus B^{\varepsilon }.$
\end{properties}

\subsection{Variographic quantities}

{The question raised in the next section is wether, for $A\subset \mathbb{R}^{2}$, 
 $
\chi^{\varepsilon } ([A]^{\varepsilon })= \chi (A)$
 for $\varepsilon $ sufficiently small, and the result depends crucially on  the regularity of $A$'s boundary.  
 A remarkable asset of  formula \eqref{eq:discrete-EC} is its nice transcription in terms of variographic tools. Let us introduce the related notation.}   For $x_{1},\dots ,x_{q},y_{1},\dots ,y_{m}\in \mathbb{R}^{2}$,   $A$ a measurable subset of $\mathbb{R}^{2}$, define the \emph{polyvariogram} of order $(q,m)$,  
\begin{align*}
\delta _{x_{1},\dots ,x_{q}}^{y_{1},\dots ,y_{m}}(A):=\Vol (  (A+x_{1})\cap\dots \cap  (A+x_{q})\cap (A+y_{1})^{c}\cap \dots \cap (A+y_{m})^{c}).
\end{align*} The variogram of order $(2,0)$  is known as the covariogram of $A$ (see \cite[Chap. 3.1]{Lan02}), and we designate here by  \emph{bicovariogram}  of $A$ the polyvariogram of order $(3,0)$. A polyvariogram of  order $(q,m)$ can be written as a linear combination of  variograms with orders $(q_{i},0)$ for appropriate numbers $q_{i}\leqslant q$. For instance for $x,y\in \mathbb{R}^{2}, A
\subset \mathbb{R}^{2}$ measurable with finite volume, we have  
\begin{align*}
\delta_0  ^{x,y}(A)
&=\delta_0 (A)-\delta_{0 ,x}(A)-\delta_{0 ,y}(A)+\delta_{0 ,x,y}(A).
\end{align*}A similar notion can be defined on $Z_{\varepsilon }$ endowed with the counting measure: for $M\subseteq Z_{\varepsilon },x_{1},\dots ,x_{q},y_{1},\dots ,y_{m}\in Z_{\varepsilon }$, 
\begin{align*}
\tilde \delta  _{x_{1},\dots ,x_{q}}^{y_{1},\dots ,y_{m}}(M):=\# (  (M+x_{1})\cap\dots \cap  (M+x_{q})\cap (M+y_{1})^{c}\cap \dots \cap (M+y_{m})^{c}).
\end{align*} 
  Lemma \ref{lm:chi-eps-phi-0} directly yields for $M\in \mathcal{A}(Z_{\varepsilon }) $,
\begin{align*}
\chi ^{\varepsilon }(M)=\tilde\delta_0 ^{-\varepsilon \u_{1},-\varepsilon \u_{2}}(M)-\tilde\delta ^{0}_{\varepsilon \u_{1},\varepsilon \u_{2}}(M).
\end{align*} 
 We will see in the next section that for a sufficiently regular set $F\subseteq \mathbb{R}^{2}$, the analogue equality $\chi (F)=\delta_0 ^{-\varepsilon \u_{1},-\varepsilon \u_{2}}(F) -\delta_{\varepsilon \u_{1},\varepsilon \u_{2}}^{0}(F)$ holds   for $\varepsilon $  small.
  
 \subsection{\EC~of $\rho $-regular sets}

  It is known in image morphology that the digital approximation of the \EC~is in general badly behaved when the set $F\in \mathcal{A}(\mathbb{R}^{2})$ possesses some inwards or outwards sharp angles, i.e. we don't have $\chi^{\varepsilon } (  [F]^{\varepsilon })\to \chi (F)$ as $\varepsilon \to 0$, the boolean model being the typical example of such a failure, see \cite[Chap. XIII - B.6]{Serra} or \cite{Sva14}.  Sets nicely behaved with respect to digitalisation are called \emph{morphologically open and closed} (MOC), or $\rho $-regular, see \cite[Chap.V-C]{Serra},\cite{Sva15}.

   Before giving the characterisation of such sets, let us introduce some morphological concepts, see for instance \cite{Lan02,Serra} for a more detailed account of mathematical morphology. 
  We state below results in $\mathbb{R}^{d}$ because  the arguments are based on purely metric considerations that apply identically in any dimension.

 \paragraph{Notation} The ball with centre $x$ and radius $r$ in the $\infty $-metric $\|\cdot \|_{\infty }$ of $\mathbb{R}^{d}$ is noted $B[x,r]$. The Euclidean ball with centre $x$ and radius $r$ is noted $B(x,r)$.
  For $r>0,A\subseteq \mathbb{R}^{d}$, define 
\begin{align*}
A^{\oplus r}:&=\{x+y: x\in A, y\in B(0,r)\}.\\
A^{\ominus r}:&=\{x\in A:B(x,r)\subseteq A\}=((A^{c})^{\oplus r})^{c}.\\
\end{align*} 
We also note $\partial A,\text{\rm{cl}}(A),\text{\rm{int}}(A)$ for resp. the topological boundary, closure, and interior of a set $A$.
Note $\S^{1}$ the unit circle in $\mathbb{R}^{2}$.\\ 
 
 Say that a closed set $F$ has an \emph{inside  rolling ball} if for each $x\in  F$, there is a closed {Euclidean} ball $B$ of radius $r$  contained in $F$  such that $x\in B$, and say that $F$ has an \emph{outside rolling ball} if $\text{cl}(F^{c})$ has an inside rolling ball.   
  
A set $F$ has \emph{reach} at least $r\geqslant 0$ if for each point $x$ at distance $\leqslant r$ from $F$, there is a unique point $y\in F$ such that $d(x,y)=d(x,F)$. We note in this case $y=\pi _{F}(x)$.
Call \emph{reach} of $F$ the supremum of the $r\geqslant  0$ such that $F$ has reach at least $r$. 
 The proposition below gathers some elementary facts about sets satisfying those rolling ball properties, the proof is left to the reader. 
 \begin{proposition} 
\label{prop:poperties-rollingball}
Let $\rho >0$ and $F$ be a closed set of $\mathbb{R}^{d}$ with an inside and an outside rolling ball of radius $\rho $. Then there is an outwards normal vector $\n_{F}(x)$ in each $x\in \partial F$. For $r\leqslant \rho $, $B_{x}:=B(x-r\n_{F}(x),r),$ resp. $B'_{x}:=B(x+r\n_{F}(x),r)$ is the unique inside, resp. outside rolling ball in $x$. Also, $\text{\rm{int}}(B'_{x})\subseteq F^{c}$.
 Furthermore, $\partial F,F$ and $F^{c}$ have reach at least $r$ for each $r<\rho $.
\end{proposition}

  We reproduce here partially the synthetic formulation of Blashke's theorem by Walther \cite{Walther}, which gives a connection between rolling ball properties and the regularity of the set.
\begin{theorem}[Blashke]
\label{th:blaschke} 
Let $F$ be a compact and connected subset of $\mathbb{R}^{d}$. 
Then for $\rho >0$  the following assertions are equivalent.\begin{itemize}
\item[(i)] $ \partial F$ is a compact $(d-1)$-dimensional $\mathcal{C}^{1}$ submanifold of $\mathbb{R}^{2}$ such that the mapping $\n_{F}(\cdot )$, which associates  to $x\in \partial F$ its outward normal vector to $F$, $\n_{F}(x),$ is $\rho ^{-1}$-Lipschitz,
\item[(ii)] $F$ has inside and outside rolling ball of radius $\rho $,
\item[(iii)] $(F^{\ominus r })^{\oplus r }=(F^{\oplus r })^{\ominus r }=F$, $r<\rho $.
\end{itemize}

\end{theorem}

  \begin{definition}
  \label{def:regular}
  \emph{Let }$F$ be a compact set of $\mathbb{R}^{d}$. Assume that $F$ has finitely many connected components and satisfies  either (i),(ii) or (iii) for some $\rho_{0}>0 $. Since the connected components of $F$ are at pairwise positive distance, each of them satisfies (i),(ii), and (iii), and therefore the whole set $F$ satisfies  (i),(ii) and (iii) for some $\rho >0$, which might be smaller than $\rho_{0} $. Such a set is said to belong to  \emph{Serra's regular class}, see the monograph of Serra \cite{Serra}. We will say that such a set is \emph{$\rho $-regular}, or simply  \emph{regular}.
   \end{definition}

\paragraph{Polyrectangles}

An aim of the present paper is to advocate the power of covariograms for computing the Euler characteristic of a regular set in the plane, we therefore need to give results that can be compared with the literature. Since many applications are  concerned with stationary random sets on the whole plane, we have to study the intersection of random regular sets with bounded windows, and assess the quality of the approximation.

To this end, call admissible rectangle of $\mathbb{R}^{2}$ any set $W=I\times J$ where $I$ and $J$ are closed (and possibly infinite) intervals of $\mathbb{R}$ with non-emtpy interior,  and note $\corners(W )$ its corners, which number is between $0$ and $4$. Then call \emph{polyrectangle} a finite union 
 $
W=\cup _{i}W_{i}
$ where each $W_{i}$ is an admissible rectangle, and for $i\neq j, \corners(W _{i})\cap \corners(W _{j})=\emptyset $. Call $\W$ the class of admissible polyrectangles. For such $W\in \W$, denote by $\text{\rm{corners}}(W)$ the elements of $\cup _{i}\text{\rm{corners}}(W)$ that lie on $\partial W$.
For $x\in \partial W\setminus \corners(W)$, note $\n_{W}(x)$ the outward normal unit vector to $W$ in $x$. We also call \emph{edge of $W$} a maximal segment of $\partial W$, i.e. a segment $[x,y]\subseteq \partial W$ that is not strictly contained in another such segment of $\partial W$ (in this case, $x,y\in \corners(W)$). Also call $\Per_{i}(W),i=1,2,$ the total length of edges where the normal vector is collinear to $\u_{i}$, 	and remark that the total perimeter of $W$ is $\Per_{1}(W)+\Per_{2}(W)$.

Using the  considerations  of Section \ref{sec:image}, it is obvious that, for $W\in \W$, $\chi (W)$ is the cardinality of $\Phi ^{\text{\rm{out}}}(W)$, which is formed by north-east outwards corners of $ \corners(W)$, minus the cardinality of $\Phi ^{\text{\rm{in}}}(W)$, the set of south-west inwards corners of $  \corners(W)$. This remark can be used to compute the mean \EC~of random sets living in $\W$.

\paragraph{Shot noise processes}
Let $\mu $ be a probability measure on $\W$, and $\nu $ a probability measure on $\mathbb{R}_{+}$ such that $\nu (\{0\})=0$.
Let $X$ be a Poisson measure on $\mathbb{R}^{2}\times \W\times \mathbb{R}_{+}$ with intensity measure $\ell\otimes\mu \otimes\nu $, where $\ell$ is Lebesgue measure on $\mathbb{R}^{2}$.  Introduce the random field
\begin{align*}
f(y)=\sum_{(x,W,m)\in X}m\mathbf{1}_{\{y\in x+W\}},y\in \mathbb{R}^{2}.
\end{align*} 
To make sure that the process is well defined a.s., assume that $$ \int_{\mathbb{R}_{+}\times \W}m\Vol(W)\nu (dm)\mu (dW)<\infty .$$ Then the law of $f$ is stationary, i.e. invariant under spatial translations.

This type of process is called a shot noise process, or moving average, and is used in image analysis, geostatistics, or many other fields. The geometry of their level sets are also the subject of a heavy literature, see for instance the recent works \cite{AST,BieDes}, but no expression seems to exist yet for the mean \EC. We provide below such a formula under some weak technical assumptions, see Section \ref{sec:proof-SN} for a proof.

\begin{theorem}
\label{thm:SN}
Let $\lambda >0,V\in \W$. Let $W_{1},M_{1},M_{2}$, independent random variables with distribution $\mu ,\nu,\nu  $ respectively.  Introduce the level set $F=\{f\geqslant \lambda \}$
and assume that $\E \Phi ^{\text{\rm{out}}}(W_{1})<\infty ,\E\Phi ^{\text{\rm{in}}}(W_{1})<\infty ,\E\Per (W_{1})<\infty $. Introduce  
\begin{align*}
p_{1}&=\P(\lambda -M_{1}\leqslant f(0)<\lambda )\\
p_{2}&=\P(\lambda -M_{1}-M_{2}\leqslant f(0) <\lambda -\max(M_{1},M_{2}) )\\
p_{2}'&=\P( \lambda -\max(M_{1},M_{2})\leqslant f(0)<\lambda ) .
\end{align*}
Then 
\begin{align}
\notag
\E \chi (F\cap V)=&\Vol(V)\left[
 p_{1}\E \chi (W_{1})+\frac{p_{2}-p_{2'}}{2}\E\Per_{1}(W_{1})\E\Per_{2}(W_{1}) 
\right]+\chi (V)\P(f(0)\geqslant\lambda )\\
\label{eq:EC-rectang}&+\frac{p_{1}}{4}\left[
\Per_{1}(V)\E\Per_{2}(W_{1})+\Per_{2}(V)\E\Per_{1}(W_{1})
\right].
\end{align}

\end{theorem}

\begin{itemize}
\item For instance, if $\mu $ is a Dirac mass in $[0,a]^{2},a>0,$ and $\nu $ is a Dirac mass in $1$, $M_{i}=1$ a.s. and $f(0)$ is a Poisson variable with parameter $a^{2}$. The volumic part of the mean  \EC~(i.e. the one multiplied by $\Vol(V)$) is therefore 
\begin{align*}
\P(f(0)=\lceil \lambda \rceil-1)\left(
1+\frac{1}{2}\left[
\lceil \lambda \rceil -1-a^{2}
\right]
\right)\end{align*}
\item For $\lambda \in (0,1)$ and $\nu $ the Dirac mass in $1$, $F$ has the law of the boolean model $M$ with random grain distributed as the $M_{i}$ and random germs given by $X$, therefore formula \eqref{eq:EC-rectang} provides, with $p_{1}=p_{2}'=\P(f(0)=0)=\exp(-\E\Vol(W_{1})),p_{2}=0$,
\begin{align*}
\E \chi (M\cap V)=&\Vol(V)p_{1}\left[
\E \chi _{1}(W_{1})- \frac{\E\Per_{1}(W_{1})\E\Per_{2}(W_{1})}{2}
\right]\\
&+\chi (V)(1-p_{1})+\frac{p_{1}}{4}\left[
\Per_{1}(V)\E\Per_{2}(W_{1})+\Per_{2}(V)\E\Per_{1}(W_{1})
\right].
\end{align*}
\end{itemize}

Coming back to the \EC~of smooth sets of $\mathbb{R}^{2}$, the following assumption needs to be in order for the restriction of a $\rho$-regular set to a polyrectangle to be topologically well behaved.

\begin{assumption}
\label{ass:inter-W}
Let $F$ be a $\rho$-regular set, and $W\in \W$. Assume that  $\partial F\cap \corners(W)=\emptyset $ and that for $x\in \partial F\cap \partial W$, $\n_{F}(x)$ is not colinear with $\n_{W}(x)$.
\end{assumption}

 If a $\rho$-regular set $F$ and a polyrectangle $W$ do not satisfy this assumption, $F\cap W$ might have an infinity of connected components, which makes the \EC~not properly defined.
  Let us prove that the digitalisation is consistent if this assumption is in order.

\begin{theorem}
\label{th:as-conv}
Let $F$ be a  $\rho$-regular set of $\mathbb{R}^{2}$, $W\in \W$ satisfying Assumption \ref{ass:inter-W}  such that $F\cap W$ is bounded. Then $F\cap W\in \A$ and  there is $\varepsilon (F,W)>0$ such that for $\varepsilon <\varepsilon(F,W) $, 
\begin{align}
\label{eq:euler-pixelized}
 \notag\chi (F\cap W)&=\chi^{\varepsilon } ([F\cap  W ] ^{\varepsilon })\\ 
 &=\sum_{x\in \varepsilon \mathbb{Z} ^{2}}\chi ^{\varepsilon }(x;F\cap W)\\
 \notag&=\varepsilon ^{-2}\int_{ \mathbb{R}^{2}}\chi ^{\varepsilon }(x;F\cap W)dx\\
\label{eq:euler-bicovariograms}
 &=\varepsilon ^{-2} \left( \delta_0 ^{-\varepsilon \u_{1},-\varepsilon \u_{2}}(F\cap W)-\delta_{\varepsilon \u_{1},\varepsilon \u_{2}}^{0}(F\cap W) \right).
\end{align} Also, $\varepsilon (F,W)=\varepsilon (F+x,W+x)$  for $x\in \mathbb{R}^{2}$.

 \end{theorem}
 
 The proof is at Section \ref{sec:proof-main}.
 
  \begin{remark}\begin{itemize}
\item [(i)] The apparent anisotropy of \eqref{eq:euler-pixelized}-\eqref{eq:euler-bicovariograms} can be removed by averaging over all pairs $\{\u_{1},\u_{2}\}$ of orthogonal unit vectors of $\mathbb{R}^{2}$.  Even though \eqref{eq:euler-bicovariograms} does not involve the discrete approximation, a direct proof not exploiting lattice approximation is not available yet, and such a proof might shed light on the  nature of the relation between covariograms and Minkowski functionals.
\item[(ii)] The fact that the \EC~of a regular set digitalisation converges to the right value is already known, see \cite[Section 6]{Sva15} and references therein, but we reprove it in Lemma 
 \ref{lm:reg-set-approx}, under a slightly stronger form.  One of the  difficulties of the proof of Theorem 
\ref{th:as-conv} is to deal with the intersection points of $\partial W$ and $\partial F$.
\item[(iii)]It is proved in Svane \cite{Sva15} that in higher dimensions, \EC~and Minkowski functionals of order $d-2$ can be approximated through isotropic analogues of formula \eqref{eq:discrete-EC}. The arguments of the proof of  Lemma 
 \ref{lm:reg-set-approx}, treating the case $F\subseteq W$, are purely metric and should be generalisable  to higher dimensions. On the other hand, dealing with boundary effects in higher dimensions might be a headache.
 \item[(iv)] It is clear throughout the proof that the value $\varepsilon (F,W)$ above for which \eqref{eq:euler-pixelized}-\eqref{eq:euler-bicovariograms} is valid is a continuous function of $\rho $, the distances between the connected components $F\cap W$,  the distances between the points of $(\partial W\cap \partial F)\cup (\corners(W)\cap F)$, and the angles between $\n_{W}(x)$ and $\n_{F}(x)$ at points $x\in \partial F\cap \partial W$.
\end{itemize}

  \end{remark}
  \begin{example} 
\label{ex:disc}
  Before giving the proof, let us give an elementary graphical illustration of \eqref{eq:euler-bicovariograms} with $F=B(0,1)$ in $\mathbb{R}^{2}$. Let $\varepsilon >0$. We note $\Gamma _{+}= F\cap (F+\varepsilon \u_{1})^{c}\cap (F+\varepsilon \u_{2})^{c}$ and $\Gamma  _{-}= F^{c}\cap (F+\varepsilon \u_{1})\cap (F+\varepsilon \u_{2})$. We should have for $\varepsilon $ small
\begin{align*}
1=\chi (F)=\varepsilon ^{-2}
 \left[
\delta_0 ^{-\varepsilon \u_{1},-\varepsilon \u_{2}}(F)-\delta^{0}_{\varepsilon \u_{1},\varepsilon \u_{2}}(F)\right],
\end{align*}$\Vol(\Gamma _{-})=\delta_{\varepsilon \u_{1},\varepsilon \u_{2}}^{0}(F)$, and $\Vol(\Gamma _{+})=\delta_0 ^{\varepsilon \u_{1},\varepsilon \u_{2}}(F)=\delta_0 ^{-\varepsilon \u_{1},-\varepsilon \u_{2}}(F)$. The notation $a,b,c,d,e,f$ designate  six distinct subsets  (see Figure \ref{fig:ball}, below) such that $\Gamma _{-}=a\cup b\cup c, \Gamma _{+}=d\cup e\cup f$. Symmetry arguments  yield that $\Vol(a)=\Vol(f), \Vol(a\cup b)=\Vol(c),\Vol(f)=\Vol(d\cup e)$,  whence 
\begin{align*}
\Vol(\Gamma _{+})-\Vol(\Gamma _{-})
&=\Vol(a)+\Vol(b)+\Vol(c)-\Vol(d)-\Vol(e)-\Vol(f)\\
&=2(\Vol(f)+\Vol(b))-\Vol(d)-\Vol(e)-\Vol(f)\\
&=\Vol(f)-\Vol(d)+2\Vol(b)-\Vol(e)\\
&=2\Vol(b).
\end{align*}The shape of $b$ is very close to that of a cube with diagonal length $\varepsilon $, i.e. with side length $2^{-1/2}\varepsilon $.
 Therefore $\Vol(b)\approx \varepsilon ^{2}/2$, which confirms $1=\chi (F)= \varepsilon ^{-2}\left[ \delta_0 ^{-\varepsilon \u_{1},-\varepsilon \u_{2}}(F)-\delta_{\varepsilon \u_{1},\varepsilon \u_{2}}^{0}(F) \right]$ (rigorously proved by Theorem \ref{th:as-conv}).
  
  \begin{figure} 
 
  \begin{tikzpicture}
    \node[anchor=south west,inner sep=0] at (0,0) {\includegraphics[scale=.5]{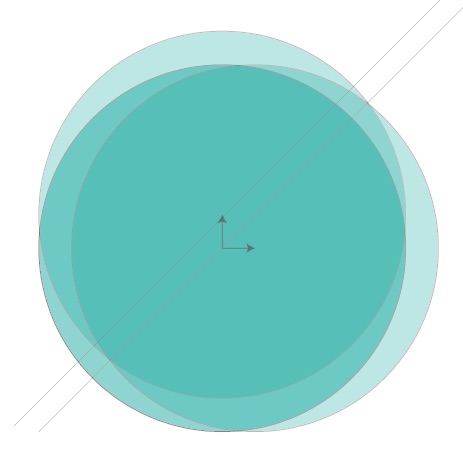}};
    \node at (4,3.5) {$0$};
        \node at (4.7,3.5) {$\varepsilon \u_{1}$};
            \node at (4,4.6) {$\varepsilon \u_{2}$};
            \node at (1.35,2.1) {$a$};
            \node at (1.65,1.7) {$b$}; 
            \node at (2.1,1.3) {$c$}; 
         
          \node at (5.85,6.5) {$d$};
            \node at (6.2,6.2) {$e$}; 
            \node at (6.5,6) {$f$};
          
\end{tikzpicture}
\caption{ \label{fig:ball}Bicovariograms of the unit disc}
  
\end{figure}
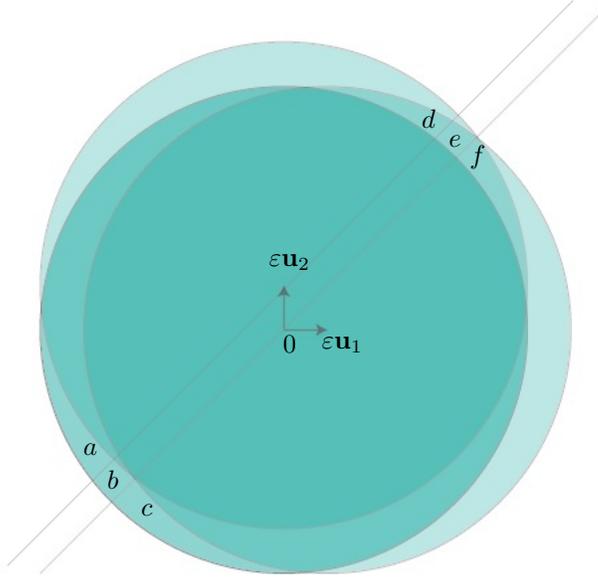

         
          


  
  
  \end{example}
   \begin{remark}\begin{itemize}
\item [(i)]
 Theorem \ref{th:as-conv} still holds if $\partial F$ intersects the outwards corner of $W$. In particular, we can drop the corner-related part of Assumption \ref{ass:inter-W} if $W$ is a rectangle. This subtlety makes the proof slightly more complicated, and such generality is not necessary in this paper.
 \item [(ii)]It should be possible to show that under the conditions of Theorem \ref{th:as-conv}, $F\cap W$ and $(F\cap W)^{\varepsilon }$ are homeomorphic, but we are only interested in the \EC~in this paper.
\end{itemize}
 \end{remark}
 
 \begin{remark}
 \label{rk:C1manifold-notlipshitz}It seems difficult to deal with $\mathcal{C}^{1}$ manifolds that don't have a Lipschitz boundary, in a general setting. Consider for instance in $\mathbb{R}^{2}$ 
\begin{align*}
A=\bigcup _{n=2}^{\infty }B((1/n,0),1/n^{2}).
\end{align*} Then $\partial A$ is a $\mathcal{C}^{\infty }$ embedded sub manifold of $\mathbb{R}^{2}$, but it has infinitely many connected components, which puts $A$  off the class $\mathcal{A}$ of sets that we consider admissible for computing the \EC.\end{remark}
 
     To have the convergence of  \EC's expectation for random regular sets, we need the domination provided by Theorem \ref{th:bound-pixel-comp} in the next section.

     \subsection{Bounding the number of  components}
     
     Taking the expectation in formula \eqref{eq:euler-pixelized} and switching with the limit  $\varepsilon \to 0$ requires a uniform upper bound in $\varepsilon $ on the right hand side. 
 For $\varepsilon $ small, \eqref{eq:euler-pixelized} consists of a lot of positive and negative terms that cancel out. Since grouping them manually  is quite intricate, this formula is not suitable for obtaining a general upper bound on $ | \chi (F^{\varepsilon })| $.  The most efficient way consists in bounding the number of components of $F^{\varepsilon }$ and $(F^{c })^{\varepsilon }$ in terms of the regularity of the set.

The result derived below is intended to be applied to $\rho$-regular sets, but we cannot make any assumption on the value of the regularity radius $\rho $, because the bound must be valid for every realisation. We therefore give an upper bound on $\#\Gamma (F^{\varepsilon })$ and $\#\Gamma ((F^{\varepsilon })^{c }) $ valid for any measurable set $F$.

The formula obtained bounds the number of connected components, which is a global quantity, in terms of occurrences of local configurations of the set, that we call \emph{entanglement points}.  Roughly, an entanglement occurs if two points of $F^{c}$ are  close but separated by a tight portion of $F$, see Figure \ref{fig:entang}. This  might create  disconnected components of $F^{\varepsilon }$ in this region although $F$ is locally connected.

To formalise this notion, let $x,y\in  Z_\varepsilon  $ grid neighbours. Introduce $\p_{x,y}\subseteq \mathbb{R}^{2}$ the closed  square with side length $\varepsilon $ such that   $x$ and $y$ are the midpoints of two opposite sides. Denote $\p_{x,y}'=\partial \p_{x,y}\setminus \{x,y\}$, which has two connected components. Then $\{x,y\}$ is an \emph{entanglement pair of points} of $F$ if $x,y\notin F$ and   $(\p'_{x,y}\cup F)\cap \p_{x,y}$ is connected.  
We call $\N_{\varepsilon }(F)$ the family of such pairs of points.

\begin{figure} 
\centering
\caption{\label{fig:entang}\emph{Entanglement point} In this example, $\{x,y\}\in \N_{\varepsilon }(F)$ because the two connected components of $\p'_{x,y}$, in lighter grey, are connected through $\gamma \subseteq (F\cap \p_{x,y})$. We don't have $\{x,y\}\in  \N_{\varepsilon }(F')$.}
\includegraphics[scale=.2]{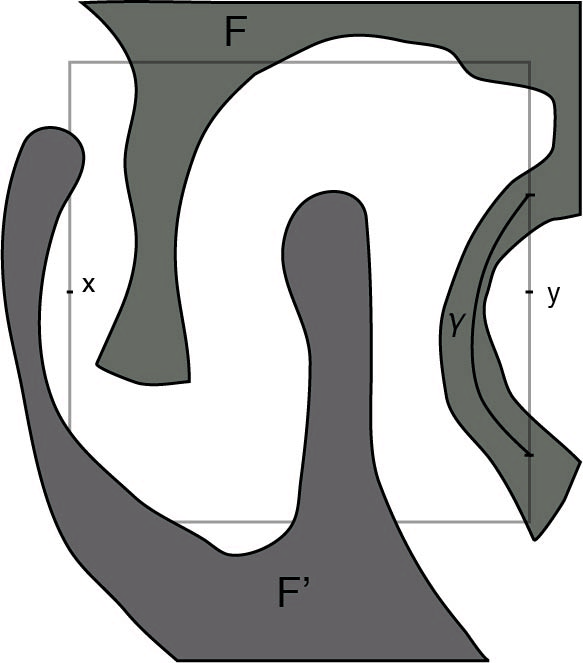}
\end{figure}

 For the boundary version, given $W\in \W,$ we also consider grid points $x,y \in   [ W\cap F]$, on the same line or column of $Z_{\varepsilon }$,  such that \begin{itemize}
\item $x,y$ are within distance $\varepsilon $ from  one of the  edges of $W$ (the same edge for $x$ and $y$)  
\item $\llparenthesis x,y\rrparenthesis \neq \emptyset $
\item $\llparenthesis x,y\rrparenthesis \subseteq [F^{c}\cap F^{\oplus \varepsilon }$].
\end{itemize}  The family of such pairs of points $\{x,y\}$ is noted $\N_{\varepsilon }'(F;W)$ .
   
  Even though $\N_{\varepsilon }(F)$ and $\N_{\varepsilon }'(F;W)$ are not points but pairs of points of $Z_{\varepsilon }$, for $A\subseteq  \mathbb{R}^{2}$, we extend the notation $\N_{\varepsilon }(F)\subseteq  A,(\text{resp. } \N_{\varepsilon }(F)\cap  A)$, to indicate  that the points of the pairs of $\N_{\varepsilon }(F)$ are contained in $A$  (resp. the collection of pairs of points from $\N_{\varepsilon }(F)$ where both points are contained in $A$), and idem for $\N_{\varepsilon }'(F,W)$. 
  
  For $\{x,y\}\in \N_{\varepsilon }(F),[x,y]\cap F\neq \emptyset $ and $x,y\in F^{c}$. Therefore $\N_{\varepsilon }(F)\subseteq \partial F^{\oplus \varepsilon }$.
    We have also  $\N'_{\varepsilon }(F,W)\subseteq  (\partial F^{\oplus \varepsilon }\cap  \partial W^{\oplus \varepsilon })$.

\begin{theorem}
\label{th:bound-pixel-comp}
Let $F$ be a bounded measurable set. Then
\begin{align}
\label{eq:bound-Gamma-eps}
 \# \Gamma (F^{\varepsilon })  \leqslant  2\#\N_{\varepsilon }(F ) + \#\Gamma (F ) 
\end{align}
and for any    $W\in \W$,
\begin{align}
\label{eq:bound-Gamma-eps-W}
  \#\Gamma ((F\cap W)^{\varepsilon }) &  \leqslant   2\#\N_{\varepsilon }(F)\cap  W^{\oplus \varepsilon }+2\#\N_{\varepsilon }'(F,W) + \#\Gamma (F\cap W)+2\#\corners(W ) .
  \end{align}

\end{theorem}

  The proof is deferred to Section \ref{sec:proof-entang}.

\begin{remark}
\label{rk:bound-euler}
Properties \ref{ppt:gauss} and \eqref{eq:bound-Gamma-eps-W} entail
\begin{align*}
\#\Gamma (((F\cap W)^{\varepsilon })^{c})&=\#\Gamma (((F\cap W)^{c})^{\varepsilon })=\#\Gamma (((F^{c}\cap W)\cup W^{c})^{\varepsilon })\\
&=\#\Gamma ((F^{c}\cap W)^{\varepsilon }\cup (W^{c})^{\varepsilon })\leqslant  \#\Gamma ((F^{c}\cap W)^{\varepsilon })+\#\Gamma ((W^{c})^{\varepsilon })
\end{align*}because adding a connected set $B$ to a given set $A$ can only decrease its number of bounded connected components, or increase it by $1$ if $B$ is bounded. It is easy to see that $\#\Gamma ((W^{c})^{\varepsilon })\leqslant \#\corners(W )$ for $\varepsilon $ sufficiently small.
It follows that
 \begin{align}
 \label{bound:EC-Ne}
\notag | \chi ((F\cap W)^{\varepsilon }) | \leqslant & \max( \# \Gamma ((F\cap W)^{\varepsilon }), \#\Gamma ((F\cap W)^{\varepsilon  })^{c})\\
 \leqslant &   3\#\corners(W )+2\max(\#\N_{\varepsilon }(F)\cap W^{\oplus \varepsilon },\#\N_{\varepsilon }(F^{c})\cap{W^{\oplus \varepsilon }})\\
\notag&+2\max(\#\N_{\varepsilon }'(F,W),\#\N_{\varepsilon }'(F^{c},W))+\max(\#\Gamma (F^{c}\cap W),\#\Gamma (F\cap W)).
\end{align}
\end{remark}

 \begin{remark} 
The boundary of a $\rho $-regular set $A$ is a $\mathcal{C}^{1}$ manifold, and can therefore be written under the form  $\partial A=f^{-1}(\{0\})$, and $\text{cl}(A)=\{f\leqslant  0\}$ for some $\mathcal{C}^{1}$ function $f$ such that $\nabla f\neq 0$ on $\partial A$ and $\|\nabla f\|^{-1}\nabla f$ is $\rho ^{-1}$-Lipschitz on $\partial A$. 
Such a function is said to be of class $\mathcal{C}^{1,1}$, see \cite{HSN}. One can bound the right hand members of \eqref{eq:bound-Gamma-eps}-\eqref{eq:bound-Gamma-eps-W} by quantities depending solely on $f$. For instance, it is proved in the companion paper \cite{LacEC2}  that in the context of Gaussian fields, $\E$ and $\lim_{\varepsilon }$ can be switched in \eqref{eq:euler-pixelized} if the derivatives of $f$ are Lipschitz and their Lipschitz constants have a finite moment of order $4+\eta $ for some $\eta >0$. 
\end{remark}

 {

  \section{Random  sets}
\label{sec:random sets}

\newcommand{\F}{\mathcal{F}}
\renewcommand{\R}{\mathcal{R}}

Let $(\Omega ,\mathcal{A},\P)$ be a complete probability space. Call $\F$ the class of closed sets of $\mathbb{R}^{2}$, endowed with the $\sigma $-algebra $\mathcal{B}$ generated by events $ {\{G\cap F  \neq \emptyset, F\in \F  \}}$, for $G$ open. A $\mathcal{B}$-measurable mapping $A:\Omega \to \F $ is called a Random Closed Set (RACS). See \cite{Mol05} for more on RACS, and equivalent definitions. The functional $\chi $ is not properly defined, and therefore not measurable, on $\F$. We introduce the subclass $\R$ of regular closed sets as defined in Definition \ref{def:regular}, and endow $\R$ with the trace   topology and Borel $\sigma $-algebra, a random regular set being a RACS a.s. in $\R$. Taking the limit in $\varepsilon \to 0$ in formula \eqref{eq:euler-pixelized} entails that $\chi $ is measurable $\R\to\mathbb{R}$ (the functionals $F\mapsto \#\Gamma (F)$ and $F\mapsto \#\Gamma (F^{c})$ are also measurable). If a random regular set $F$ satisfies a.s. Assumption \ref{ass:inter-W} with some $W\in \W$, then $\#\Gamma (F\cap W),\#\Gamma ((F\cap W)^{c})$ and $\chi (F\cap W)$ are also measurable quantities.

Introduce the support $\text{\rm{ supp}}(A)$ of a RACS $A$ as the smallest closed set $K$ such that $ \P(A\subseteq K )=1$. Mostly for simplification purpose, we will assume whenever relevant that $\text{\rm{supp}}(A)$ is bounded. \\

 It is easy to derive a result giving the mean \EC~as the limit of the right hand side expectation in \eqref{eq:euler-pixelized}  by combining Theorems \ref{th:as-conv} and \ref{th:bound-pixel-comp}. We treat below the example of stationary random sets, i.e. which law is invariant under the action of the translation group. A non-trivial stationary RACS $F$ is a.s. unbounded, therefore we must consider the restriction of $F$ to a bounded window $W$.
 The main issue is to handle boundary terms stemming from the intersection. They involve the perimeter of $W$ and the specific perimeter of $F$.
We introduce the square perimeter $\Per_{\infty }$ of a measurable set $A$ with finite Lebesgue measure by the following. Note $\mathcal{C}^{1} _{c}$ the class of compactly supported functions of class $\mathcal{C} ^{1}$ on $\mathbb{R}^{2}$, and define $
\Per_{\infty }(A)=\Per_{\u_{1}}(A)+\Per_{\u_{2}}(A)$, where 
\begin{align*}
\Per_{\u}(A)=\sup_{\varphi \in \mathcal{C}^{1}_{c}: | \varphi (x) | \leqslant  1}\int_{A}\langle \nabla \varphi (x),\u   \rangle dx=\Per_{-\u}(A),\:\u\in \S^{1},
\end{align*}so that we also have the expression
\begin{align*}
\Per_{\infty }(A)=\sup_{\varphi \in \C^{1}_{c}:\|\varphi (x)\|_{\infty }\leqslant 1}\int_{A}\text{div}(\varphi )(x)dx.
\end{align*}
The classical variational perimeter is defined by
\begin{align*}
\Per(A)=\frac{1}{4}\int_{\S^{1}}\Per_{\u}(A)\sigma (d\u),
\end{align*} with the renormalized Haar measure $\sigma $ on the unit circle, it satisfies $\Per(A)\leqslant \Per_{\infty }(A)\leqslant\sqrt{2} \Per(A)$.
We have for instance for the square $W=[0,a]^{2}$, 
$
\Per_{\infty }(W)=4a
$
and for a ball $B$ with unit diameter in $\mathbb{R}^{2}$, $\Per_{\infty }(B)=4$.  
 
{
 
It is proved in \cite[(1)]{Galerne} that for any  bounded measurable set $A$,
\begin{align*}
\Per_{\u}(A)&=2 \lim_{\varepsilon \to 0}\varepsilon ^{-1}  \delta_0  ^{\varepsilon \u}(A)=2 \lim_{\varepsilon \to 0}\varepsilon ^{-1}  \delta_0  ^{-\varepsilon \u}(A)
\end{align*}
and in \cite[Proposition 16-(8)]{Galerne} that for any  RACS $F$ with compact support  
\begin{align}
\notag\E\Per_{\infty }(F)&
=2\sum_{i=1}^{2}\lim_{\varepsilon \to 0}\varepsilon ^{-1}\E \delta_0  ^{\varepsilon \u_{i}}(A)
=2\sum_{i=1}^{2}\lim_{\varepsilon \to 0}\varepsilon ^{-1}\E \delta_0  ^{-\varepsilon \u_{i}}(A)\\
\label{eq:Fubini-Per-infinity}&=2\sum_{i=1}^{2}\lim_{\varepsilon \to 0}\varepsilon ^{-1}\int_{\mathbb{R}^{2}}\P(x\in A,x+\varepsilon \u_{i}\notin A)dx.
\end{align} 
An important feature of this formula is that the mean perimeter can be deduced from the second order marginal distribution $(x,y)\mapsto \P(x,y\in F)$ in a neighbourhood of the diagonal $\{(x,x);x\in \mathbb{R}^{2}\}$. The formulas above provide a strong connection between the perimeter, called first-order Minkowski functional in the realm of convex geometry, the covariogram, and the second order marginal of a random set.

It is difficult not to notice the analogy featured by the material contained in this paper. The results in the present section emphasise the connection  between the \EC, Minkowski functional of order $0$, and the bicovariogram, a functional that can be expressed in function of the third order marginal of a random set, in a neighbourhood of the diagonal $\{(x,x,x);x\in \mathbb{R}^{2}\}$.

The results below have been designed to provide an application  in the context of random functions excursions, a field which has been been the subject of intense research recently, see the references in the introduction. We show in the companion paper \cite{LacEC2} that the quantities in \eqref{eq:hyp-expect-finite} can be bounded by finite quantities under some regularity assumptions on the underlying field, and give explicit mean \EC~for some stationary Gaussian fields.

 Say that a closed set $F$    is  locally regular  if for any compact set $W$, there is a $\rho$-regular  set $F'$ such that  $F\cap W=F'\cap W$.  
 \begin{proposition}
\label{prop:stationary-set}
 Let $F$ be a stationary random closed set, a.s. locally regular, and $
 W\in \W$  bounded. Assume that   the following local expectations are finite: 
\begin{align}
\label{eq:hyp-expect-finite}
\E \sup_{0\leqslant \varepsilon \leqslant 1 }\#\N_{\varepsilon }(F)\cap W,\;&\E \sup_{0\leqslant \varepsilon \leqslant 1 }\#\N_{\varepsilon }(F^{c})\cap W,\E\sup_{0\leqslant \varepsilon \leqslant 1 }\N_{\varepsilon }'(F,W),\;\E\sup_{0\leqslant \varepsilon \leqslant 1 }\N_{\varepsilon }'(F^{c},W),\\
\notag&\E \#\Gamma (F\cap W),\;\E\#\Gamma (F^{c}\cap W).
\end{align}  
  Then $ \E | \Gamma (F\cap W) |<\infty, \;\E | \Gamma ((F\cap W)^{c}) | <\infty   $, and the following limits are finite  
  \begin{align*}
\overline{\chi }(F):& = \lim_{\varepsilon \to 0}\varepsilon ^{-2}\left[ \P(0\in F,\varepsilon \u_{1}\notin F,\varepsilon \u_{2}\notin F)-\P(0\notin F,-\varepsilon \u_{1}\in F,-\varepsilon \u_{2}\in F)  \right],\\  
\overline{\Per_{\u_{i}} }(F):&=2\lim_{\varepsilon \to 0}\varepsilon ^{-1}\P(0\in F,\varepsilon \u_{i}\notin F) , i=1,2,\\
\overline{\Vol}(F):&= \P(0\in F).
\end{align*}We also have, with $\overline{\Per_{\infty }}(F)=\sum_{i=1}^{2}\overline{\Per_{\u_{i}}}(F)$,
\begin{align}
\label{eq:stationary-expectation-EC}
\E \chi (F\cap W)&=\Vol(W)\overline{\chi (F)}+\frac{1}{4}\Big(\Per_{\u_{2}}(W)\overline{\Per_{\u_{1} }}(F)\\
\notag&\hspace{4cm}+\Per_{\u_{1}}(W)\overline{\Per_{\u_{2} }}(F)\Big)+\chi (W)\overline{\Vol}(F)\\
\label{eq:stationary-expectation-Per}\E \Per_{\infty } (F\cap W)&=\Vol(W)\overline{\Per_{\infty }} (F)+\Per(W)\overline{\Vol}(F) \\
\label{eq:stationary-expectation-Vol}
\E \Vol(F\cap W)&=\Vol(W)\overline{\Vol(F)}.
\end{align}

 \end{proposition}
 
{
\begin{proof}
Let us first compute the mean volume and perimeter.
A straightforward application of Fubini's theorem gives \eqref{eq:stationary-expectation-Vol}. Assume for now that $\overline{\Per_{\u_{i}}}(F)$ exists, for $i=1,2$ (proved later).
We have by \eqref{eq:Fubini-Per-infinity}
\begin{align*}
\E\Per_{\infty }&(F\cap W)=2\sum_{i=1}^{2}\lim_{\varepsilon \to 0}\varepsilon ^{-1}\int_{\mathbb{R}^{2}}\P(x\in F\cap W,x+\varepsilon \u_{i}\notin F\cap W)dx\\
&=2\sum_{i=1}^{2}\lim_{\varepsilon \to 0}\varepsilon ^{-1}
 \left[
  \int_{x\in W:x+\varepsilon \u_{i}\in W}\hspace{-1cm}\P(x\in F,x+\varepsilon \u_{i}\notin F)dx
+
  \int_{x\in W:x+\varepsilon \u_{i}\notin W}\hspace{-1cm}\P(x\in F)dx
 \right]\\
 &=2\sum_{i=1}^{2}\lim_{\varepsilon \to 0}
 \left[
  (\Vol(W)+o(1))\varepsilon ^{-1}\P(0\in F,\varepsilon \u_{i}\notin F)dx+
\P(0\in F)\varepsilon ^{-1}\delta_0 ^{-\varepsilon \u_{i}}(W)
 \right]\\
 &=2\sum_{i=1}^{2}\left(\Vol(W)\frac{1}{2}\overline{\Per_{\u_{i}}}(F)+\overline{\Vol}(F)\frac{1}{2}\Per_{\u_{i}}(W)\right),
\end{align*}which gives \eqref{eq:stationary-expectation-Per}.

 Theorem \ref{th:as-conv} yields that \eqref{eq:euler-pixelized} holds if a.s. $\n_{F}(x)\notin\n_{W}(x),x\in \partial F\cap \partial W$, and $\partial F$ almost never touches a corner of $W$. 
 Let us prove that it is so.
Call 
\begin{align*}
\Theta =\{x_{[2]}\in \mathbb{R}:\n_{F}(x_{[1]},x_{[2]})=\pm\u_{2}\text{ for some }x_{[1]}\in \mathbb{R}\text{\rm{  such that }}(x_{[1]},x_{[2]})\in \partial F\}.
\end{align*} We use below Sard's Theorem to prove that $\Theta $ is a.s. negligible. 

Since $\partial F$ is locally a $\mathcal{C}^{1}$ manifold, there is a countable family of  real bounded intervals $I_{k}$ and open sets $(\Omega _{k})$  covering $\partial F$ such that on each $\Omega _{k}$, $\partial F$ can  be represented via the implicit function theorem by the graph of a $\mathcal{C}^{1}$ function $g_{k}:I_{k}\subseteq \mathbb{R}\to \mathbb{R}$. 
Let $K$ be the countable family of $k$ such that  $\partial F\cap \Omega _{k}=\{(t,g_{k}(t));\;t\in I_{k}\}\cap \Omega _{k}$. In particular, $\n_{F}(\cdot )$ is not collinear to $\u_{1}$ on $\Omega _{k},k\in K$.  

Sard's theorem entails that for $k\in K$, the set of critical values of $g_{k}$ is negligible, whence
\begin{align*}
\Theta _{k}:=\{y:\;y=g_{k}(t)\text{ for some }t\in I_{k}\text{ such that }g_{k}'(t)=0\},
\end{align*}has $0$ Lebesgue measure. For $k\notin K$, $\n_{F}(\cdot )$ is not collinear to $\u_{2}$ on $\Omega _{k}$, otherwise the IFT could not be applied, whence $\Theta \cap \Omega _{k}=\emptyset $. Therefore $\Theta \subseteq \cup _{k\in K}\Theta _{k}$ has also a.s. $0$ Lebesgue measure. Fubini's theorem then yields, noting $\ell$ the $1$-dimensional Lebesgue measure,
\begin{align*}
0=\E\ell(\Theta )=\int_{\mathbb{R}}\P(\exists x_{[1]}\in \mathbb{R}:(x_{[1]},x_{[2]})\in \partial F,\n_{F}(x_{[1]},x_{[2]})=\pm\u_{2})dx_{[2]},
\end{align*} whence by stationarity $\P(\exists x_{[1]}\in \mathbb{R}:(x_{[1]},x_{[2]})\in \partial F,\n_{F}(x_{[1]},x_{[2]})=\pm\u_{2})=0$ for all $x_{[2]}\in \mathbb{R}$. Since the set of $x_{[2]}\in \mathbb{R}$ such that, for some $x_{[1]}\in \mathbb{R}$, $(x_{[1]},x_{[2]})\in \partial W$ and $\u_{2}\in\n_{W}(x_{[1]},x_{[2]})$ is finite (it corresponds to the second coordinates of horizontal edges of $W$), we have that a.s. $\n_{F}(x)\neq \pm\u_{2}$ for every $x\in \partial W\cap \partial F$ such that $\pm\u_{2}\in\n_{W}(x)$. With an exact similar reasoning, the same statement with $\u_{1}$ instead of $\u_{2}$ holds.  We therefore proved that a.s., for $x\in \partial F\cap \partial W,\n_{F}(x)$ is not colinear with $\n_{W}(x)$, as is required in Assumption \ref{ass:inter-W}.
 
 Since $\partial F$ is a.s. a $(d-1)-$dimensional manifold, it has a.s. vanishing $2$-dimensional Lebesgue measure, whence by Fubini's Theorem, the probability that one of the   corners of $W$ belongs to $\partial F$   is $0$. It follows that Assumption \ref{ass:inter-W} is satisfied a.s., whence the a.s. convergence of Theorem \ref{th:as-conv} holds. Lebesgue's Theorem then yields, thanks to the domination \eqref{eq:hyp-expect-finite},
  using also \eqref{eq:euler-pixelized},
\begin{align*}
\E \chi &(F\cap W) =\lim_{\varepsilon \to 0}\sum_{x\in \varepsilon \mathbb{Z} ^{2}}\E\chi ^{\varepsilon }(x;F\cap W)\\
&=\lim_{\varepsilon \to 0}\Bigg[ \sum_{x\in  [W]:x+\varepsilon \u_{1}\in W,x+\varepsilon \u_{2}\in W }\large[ \P(x\in F,x+\varepsilon \u_1\notin F,x+\varepsilon \u_2\notin F)\\
&\hspace{5.5cm}-\hspace{-1cm} \sum_{x\in  [W]:x-\varepsilon \u_{1}\in W,x-\varepsilon \u_{2}\in W }\hspace{-1cm}\P(x\notin F,x-\varepsilon \u_1\in F,x-\varepsilon \u_2\in F) \large]\\
&\hspace{2.5cm} +\hspace{-1cm}\sum_{x\in  [W]:x+\varepsilon \u_{1}\in W,x+\varepsilon \u_{2}\notin W }
\hspace{-1cm}
\P(x\in F,x+\varepsilon \u_1\notin F) 
+
\hspace{-1cm}\sum_{x\in [W]:x+\varepsilon \u_{1}\notin W,x+\varepsilon \u_{2}\in W}
\hspace{-1.5cm}
\P(x\in F,x+\varepsilon \u_2\notin F)\\
&\hspace{2cm}+\hspace{-1cm}\sum_{x\in [W]:\Phi ^{\varepsilon }_{\text{out}}(x,W)=1}\P(x\in F)-\sum_{x\in Z_{\varepsilon }:\Phi ^{\varepsilon }_{\text{in}}(x,W)=1}\P(x-\varepsilon \u_{1}\in F;x-\varepsilon \u_{2}\in F)
 \Bigg]\\
&=\lim_{\varepsilon \to 0}\Big[ \varepsilon ^{-2}(\Vol(W)+o(1))\Big[ \P(0\in F,\varepsilon \u_1\notin F,\varepsilon \u_2\notin F)\\
&\hspace{7cm}-\P(0\notin F,-\varepsilon \u_1\in F,-\varepsilon \u_2\in F) \Big]\\
&\hspace{2cm}+\varepsilon ^{-1}\left(\frac{\Per_{\u_{2}}(W)}{2}+o(1)\right)\P(0\in F,\varepsilon \u_1\notin F)\\
&\hspace{2cm}+\varepsilon ^{-1}\left(\frac{\Per_{\u_{1}}(W)}{2}+o(1)\right)\P(0\in F,\varepsilon \u_2\notin F) \\
&\hspace{1.3cm}+(\Phi ^{\varepsilon }_{\text{out}}(W)+o(1))\P(0\in F)
-(\Phi ^{\varepsilon }_{\text{in}}(W)+o(1))(\P(-\varepsilon \u_{1}\in F,-\varepsilon \u_{2}\in F))\Big].
\end{align*}
Define
\begin{align*}
 \overline{\chi }^{\varepsilon }(F)&= \varepsilon ^{-2}\left[ \P(0\in F,\varepsilon \u_1\notin F,\varepsilon \u_2\notin F)-\P(0\notin F,-\varepsilon \u_1\in F,-\varepsilon \u_2\in F)  \right]\\
\overline{ \Per}^{\varepsilon }_{\u_{i}}(F)&=2\varepsilon ^{-1}\P(0\in F,\varepsilon \u_{i}\notin F)\\
\alpha _{\varepsilon }&=\P(-\varepsilon \u_{1}\in F,-\varepsilon \u_{2}\in F)-\P(0\in F).
\end{align*}
Then the previous expression becomes, using { $\chi (W)=\Phi _{\text{out}}^{\varepsilon }(W)-\Phi _{\text{in}}^{\varepsilon }(W)$} for $\varepsilon $ small enough  (easily deduced from \eqref{eq:discrete-EC}),
\begin{align*}
\E \chi (F\cap W)=\lim_{\varepsilon \to 0}\Big[ \Vol(W)\overline{\chi }^{\varepsilon }(F)+\frac{1}{4} \left(
\Per_{\u_{1}}(W)\overline{\Per}^{\varepsilon }_{\u_{2}}(F)+
\Per_{\u_{2}}(W)\overline{\Per}^{\varepsilon }_{\u_{1}}(F)
\right)\\
+\chi (W)\overline{\Vol}(F)-\alpha _{\varepsilon }\Phi _{\text{in}}(W) \Big].
\end{align*}

Let $W,W_{1},W_{2}\in \W$ simply connected such that $\Vol(W)=\Vol(W_{1})=\Vol(W_{2}),\Per_{\u_{i}}(W)=\Per_{\u_{i}}(W_{i})\neq \Per_{\u_{i}}(W_{i'}), i=1,2$. Subtracting the expression above for $W$ and $W_{i}$ for each $i=1,2$ yields that $\overline{\Per_{\u_{i} }}^{\varepsilon }(F)$ has a limit, as announced.
Then
\begin{align*}
\alpha _{\varepsilon }&\leqslant \left| \P(0\in F,-\varepsilon \u_{1}\in F,-\varepsilon \u_{2}\in F)-\P(0\in F)+\P(0\notin F,-\varepsilon \u_{1}\in F,-\varepsilon \u_{2}\in F) \right| \\
&\leqslant \P(0\in F,-\varepsilon \u_{1}\notin F\text{ or }-\varepsilon \u_{2}\notin F)+\max(\P(0\notin F,-\varepsilon \u_{1}\in F),\P(0\notin F,-\varepsilon \u_{2}\in F))\\
&\leqslant \P(0\in F,-\varepsilon \u_{1}\notin F)+\P(0\in F,-\varepsilon \u_{2}\notin F)+\P(0\notin F,-\varepsilon \u_{1}\in F)+\P(0\notin F,-\varepsilon \u_{2}\in F).
\end{align*}
Since $F$ is stationary, $\P(0\in F,\varepsilon \u_{i}\notin F)=\P(0\in F,-\varepsilon \u_{i}\notin F)=\P(0\notin F,\varepsilon \u_{i}\in F)$. Therefore,
\begin{align*}
\P(0\notin F,\varepsilon \u_{1}\in F)+\P(0\notin F,\varepsilon \u_{2}\in F)\leqslant  \varepsilon\lim_{\varepsilon \to 0}\varepsilon ^{-1}(\overline{\Per_{\u_{1} }}^{\varepsilon }+\overline{\Per_{\u_{2} }}^{\varepsilon }(F))
\end{align*}
and $\alpha _{\varepsilon }\to 0$. 
It then follows that $\chi ^{\varepsilon }(F)$ has a finite limit, which concludes the proof. 
\end{proof}
}

 \section{Proofs}
 
 \subsection{Proof of Theorem \ref{thm:SN}}
  \label{sec:proof-SN}

Fubini's theorem and $\E\Phi ^{\text{\rm{out}}}(W_{1})<\infty ,\E \Phi ^{\text{\rm{in}}}(W_{1})<\infty $ entail that $F\cap V\in \W$ a.s.. Using $\E\Vol(W_{1})<\infty $ also yields that $\Phi ^{\text{\rm{out}}}(F\cap V),\Phi ^{\text{\rm{in}}}(F\cap V)$  have finite expectation. Therefore, $\E\chi (F\cap V)=\E\Phi ^{\text{\rm{out}}}(F\cap V)-\E\Phi ^{\text{\rm{in}}}(F\cap V).$ 
  Let us enumerate the (random) points of $F\cap V$ that can belong either to $\Phi ^{\text{\rm{out}}}(F\cap V)$ or to $\Phi ^{\text{\rm{in}}}(F\cap V)$. \begin{itemize}
\item A point $y$ of 
$\Phi ^{\text{\rm{out}}}(V)$ (resp. $\Phi ^{\text{\rm{in}}}(V)$) indeed belongs to $\Phi ^{\text{\rm{out}}}(V\cap F)$ (resp. $\Phi ^{\text{\rm{in}}}(F\cap V)$) iff $f(y)\geqslant\lambda $.
\item Given $(x,W,m)\in X$, a point  $y\in \Phi ^{\text{\rm{out}}}(x+W)\cap V$ (resp. $\Phi ^{\text{\rm{in}}}(x+W)\cap V$) is in $\Phi ^{\text{\rm{out}}}(F\cap V)$ (resp. $\Phi ^{\text{\rm{in}}}(F\cap V)$) iff $f(y)\geqslant\lambda $ and points in the neighbourhood of $y$ that are not in $W+x $ are not in $F$ either, i.e. $f(y)-m< \lambda .$
\item  A point $y\in \partial F\cap \partial V$ is in $\Phi ^{\text{\rm{out}}}(F\cap V)$ if $f(y)\geqslant\lambda ,f(y)-m<\lambda $, where $m$ is the unique mass such that $y\in \partial (W+x)$ for some $(x,W,m)\in X$, and the boundaries of $(W+x)$ and $V$ indeed form a North-East outwards angle in $y$.  
\item The last possibility is a point $y\in \partial (W+x)\cap \partial (W'+x')$ for two distinct triples $(x,W,m),(x',W',m')\in X$. In this case, $y\in \Phi ^{\text{\rm{out}}}(F\cap V)$ if $f(y)\geqslant\lambda ,f(y)-\min(m,m')<\lambda $, and $y\in \Phi ^{\text{\rm{in}}}(F\cap V)$ if $\lambda > f(y)-m-m';\lambda \leqslant f(y)-\max(m,m')$.
\end{itemize} 
 
We have, using the stationarity of $f,$
\begin{align*}
\E\#\Phi ^{\text{\rm{out}}}(V)\cap F=\P(f(0)\geqslant\lambda )\#\Phi^ \text{\rm{out}}(V),\hspace{1cm}\E\#\Phi ^{\text{\rm{in}}}(V)\cap F=\P(f(0)\geqslant\lambda )\#\Phi ^{\text{\rm{in}}}(V).
\end{align*}
Then,   using Mecke's formula,
\begin{align*}
 \E\sum_{(x,W,m)\in X}&\sum_{y\in \Phi ^{\text{out}}(W+x)\cap V}\mathbf{1}_{\{f(y)-m<\lambda \leqslant f(y) \}}\\
&=\int_{\mathbb{R}^{2}\times \W\times \mathbb{R}_{+}}\E\sum_{y\in \Phi ^{\text{\rm{out}}}(W+x)\cap V}\mathbf{1}_{\{f(y)<\lambda \leqslant f(y)+m\}}dx\nu (dm)\mu (dW)\\
&=\underbrace{\P(f(0)<\lambda \leqslant f(0)+M_{1})}_{p_{1}}\E\sum_{y\in \Phi ^{\text{\rm{out}}}(W_{1})}\int_{\mathbb{R}^{2} }\mathbf{1}_{\{y\in V-x\}} dx\mu (dW)\\ 
&=p_{1}\Vol(V)\E\#\Phi ^{\text{\rm{out}}}(W_{1}),\\
\end{align*}
and a similar computation holds to treat the points of the $\Phi ^{\text{\rm{in}}}(W+x)$.
For the next term, for $W\in \W$, note $\partial ^{N}W$ the edges of $W$ facing North, $\partial ^{E}W$ those facing East, and $\partial ^{+}W=\bigcup _{e\in \partial ^{N}W\cup \partial ^{E}W}e$. Then
\begin{align*}
\E\Phi ^{\text{\rm{out}}}&(V \cap F)\cap \partial V = 
 \E\sum_{(x,W,m)\in X}\sum_{y\in \partial ^{+}V\cap \partial ^{+}(W+x)}\mathbf{1}_{\{f(y)-m<\lambda \leqslant f(y) \}}\\
 &= \int_{\mathbb{R}^{2} \times \mathbb{R}_{+}}\E \sum_{y\in \partial ^{+}(W_{1}+x)\cap \partial^{+} V}\mathbf{1}_{\{f(y) <\lambda \leqslant f(y) +m\}}dx \nu (dm)\\
 &=p_{1}\E\Bigg[
\sum_{e\in \partial ^{N}W_{1},f\in \partial ^{E}V} \int_{  \mathbb{R}^{2}}\mathbf{1}_{\{e\cap (f-x)\neq \emptyset \}} dx\\
&\hspace{4cm}+\sum_{e\in \partial ^{E}W_{1},f\in \partial ^{N}V} \int_{  \mathbb{R}^{2}}\mathbf{1}_{\{e\cap (f-x)\neq \emptyset \}} dx\Bigg]\\
&=p_{1}\E\left[
\sum_{e\in \partial ^{N}W_{1},f\in \partial ^{E}V} | e |  | f | +\sum_{e\in \partial ^{E}W_{1},f\in \partial ^{N}V} | f |  | e | 
\right]\\
&=p_{1}\frac{1}{4}\left[\Per_{1}(V)
\E\Per_{2}(W_{1})+\Per_{2}(V)\E\Per_{1}(W_{1})
\right].
 \end{align*}
It remains to compute the term stemming from the intersection of distinct grains. The expected number of such points in $\Phi ^{\text{\rm{out}}}(F\cap V)$ is
\begin{align*}
 \E&\sum_{(x,W,m)\neq (x',W',m')\in X}\sum_{y\in \partial ^{+}(W+x)\cap \partial ^{+}(W'+x')\cap V}\mathbf{1}_{\{f(y)-\min(m,m')<\lambda \leqslant f(y) \}}\\
=& \E\int_{(\mathbb{R}^{2}  )^{2}}\sum_{y\in \partial ^{+}(W_{1}+x)\cap \partial ^{+}(W_{2}+x')\cap V}\underbrace{\P(f(y) +\max(M_{1},M_{2})<\lambda \leqslant f(y)+M_{1}+M_{2})}_{p_{2}}dxdx'\\
 =& p_{2}
 \underbrace{\E\int_{\mathbb{R}^{2} }\left[
\int_{\mathbb{R}^{2}}\#\Big(
\partial ^{+}(W_{1}+x)\cap \partial ^{+}(W_{2}+x')\cap V
\Big)dx'
\right]dx}_{I_{N}+I_{E}}
\end{align*}
where
\begin{align*}
I_{N}& = \E\int_{\mathbb{R}^{2}}\sum_{e\in \text{\rm{edges}}^{N}(W_{1}+x)}\sum_{f\in \text{\rm{edge}}^{E}(W_{2})}\int_{\mathbb{R}^{2}}\mathbf{1}_{\{e\cap V\cap (f+x')\neq \emptyset \}}dx'dx \\
& = \E
 \int_{\mathbb{R}^{2}}\sum_{e\in \text{\rm{edges}}^{N}(W_{1}+x),f\in \text{\rm{edge}}^{E}(W_{2})} {  | e\cap V |  |  f |  }dx\\
 &=\frac{\E\Per_{2}(W_{2})}{2}\E\sum_{e\in \text{\rm{edges}}^{N}(W_{1})}\int_{\mathbb{R}^{2}} | e\cap (V-x ) | dx= 
\frac{1}{4}\Vol(V)\E\Per_{2}(W_{1})\E\Per_{1}(W_{1}),
\end{align*}
and $I_{E}=I_{N}$ is computed similarly. With an analogue computation, the expected number of such points contributing to $\Phi ^{\text{\rm{in}}}(F\cap V)$ is $\frac{1}{2}p_{2}'\Vol(V)\E\Per_{1}(W_{1})\E\Per_{2}(W_{1})$.
 Finally, 
\begin{align*}
\E \chi (F\cap V) =&\Vol(V)(p_{1}\E\chi (W_{1})+\frac{p_{2}-p_{2}'}{2}\E\Per_{1}(W_{1})\E\Per_{2}(W_{2}) )\\
&+(\#\Phi ^{\text{\rm{out}}}(V)-\Phi ^{\text{\rm{in}}}(V))\P(f(0)\geqslant\lambda )\\
&+\frac{p_{1}}{4}\left[
\Per_{1}(V)\E\Per_{2}(W_{1})+\Per_{2}(V)\E\Per_{1}(W_{1})
\right].
\end{align*}

 \subsection{Proof of Theorem \ref{th:as-conv}}
 \label{sec:proof-main}
 
 Let $\rho >0$ such that $F$ satisfies (i),(ii) and (iii) in Theorem \ref{th:blaschke} (see Definition \ref{def:regular}).  The proof is divided in several steps, presented as Lemmas.
  
  \begin{lemma}$F\cap W\in \mathcal{A}$ \end{lemma}
\begin{proof}  Since every connected component of $F$ contains at least one ball of radius $\rho $ and one cannot pack an infinity of such balls in $F\cap W$, the number of connected components of $F$ contained in $W$ is finite. A similar reasoning holds for $F^{c}$, using Proposition \ref{prop:poperties-rollingball}.
  
   We must still prove that $F$ and $F^{c}$ have a finite number of connected components touching the boundary, i.e. prove that $\partial F\cap \partial W$ is finite. If it is not so, there is a countable infinite family of disjoint segments $I_{n}=[x_{n},y_{n}] $ of $F\cap \partial W$, with $x_{n},y_{n}\in \partial F$, and an infinite number of them lie on a given edge of $W$ because $W$ has a finite number of edges. Up to applying rotations and symmetries to $F$, assume  that $(x,y)$ is horizontal and  that the ordering is such that $x_{[1]}\leqslant y_{[1]}<x_{[2]}\leqslant y_{[2]}<\dots $. Then, since $\n_{F}(x_{n})_{[1]}\leqslant 0$ and $\n_{F}(y_{n})_{[1]}\geqslant 0$, there is a common limit $z\in [x,y]\cap \partial F$ of the $x_{n}$ and $y_{n}$ such that $\n_{F}(z)_{[1]}=0$, contradicting Assumption \ref{ass:inter-W}, using the continuity of $\n_{F}(\cdot )$. It follows that $F\cap W\in \A$.
   \end{proof}
   
   \begin{lemma} \label {lm:PhiX}For $\varepsilon $ sufficiently small, $\Phi _{X}^{\varepsilon }([F\cap W])=0$.\end{lemma}
\begin{proof} 
  Let $\varepsilon <  \sqrt{2}(\sqrt{2}-1)\rho$. Elementary geometric considerations yield that we cannot find balls $B_{i},i=1,\dots ,4$ such that \begin{itemize}
\item $0\in B_{1},\varepsilon \u_{1}\in B_{2},\varepsilon \u_{2}\in B_{3},\varepsilon (\u_{1}+\u_{2})\in B_{4}$
\item Each ball has radius $\geqslant \rho $
\item $B_{1}\cap B_{2}=\emptyset ,B_{2}\cap B_{3}=\emptyset ,B_{3}\cap B_{4}=\emptyset ,B_{4}\cap B_{1}=\emptyset $.
\end{itemize}

The inside and outside  rolling ball conditions then yield that $\Phi _{X}^{\varepsilon }(0,[F])=0$. Reasoning similarly in every point  and summing yields $\Phi _{X}^{\varepsilon }([F])=0$. 
To prove that $\Phi _{X}^{\varepsilon }([F\cap W])=0$, let us first remark that due to Assumption \ref{ass:inter-W}, for $\varepsilon $ small enough, the points of $  \partial F\cap \partial W$ are at distance more than $2\varepsilon $ from $\corners(W)$. 
Therefore the intersection of $F$ with $W$ cannot add $X$-configurations,  $\Phi _{X}^{\varepsilon }([F\cap W])=0$.
\end{proof}

We now proceed to show that $F\cap W$ and $(F\cap W)^{\varepsilon }$ have the same \EC.  The proof is divided in two parts. We first prove {\bf (i)} the result under the assumption that $F\subseteq W$, avoiding boundary issues, and then {\bf (ii)} complete the proof without this assumption.

{\bf (i)} Let us first assume that $F\subseteq  W$, or equivalently that $F$ is bounded and $W=\mathbb{R}^{2}$.
The following result shows that,  if boundary problems are put aside, a regular set and its Gauss approximation are homeomorphic when $\varepsilon $ is small enough. Even though we could not locate the result under this particular form, the fact that the topology of a $\rho$-regular set is preserved by digital approximation is already known in image analysis, see \cite{Sva15} and references therein. We defer the proof to the Appendix.
 \begin{lemma}
 \label{lm:reg-set-approx}
Let $F$ be a  $\rho  $-regular bounded set, and 
let  $\varepsilon < \rho /\sqrt{2}$ be such that given any two distinct connected components $C,C'$ of $F$, $d(\partial C,\partial C')>4\varepsilon $.  Then $F$ is homeomorphic to $F^{\varepsilon }$.
\end{lemma}

 \begin{proof}

 We start with a lemma.
\begin{lemma} 
\label{lm:homeomorphic}
Let $A$ be a connected set, $r>0$ and $0<\varepsilon \leqslant 2^{-1/2}r$. Then $[A^{\oplus r}]^{\varepsilon }$ is connected in $Z_{\varepsilon }$, whence $(A^{\oplus r})^{\varepsilon }$ is connected in $\mathbb{R}^{2}$.
\end{lemma}
\begin{proof} 
Let $x,y\in \mathbb{R}^{2}$ such that $\|x-y\|\leqslant r$, $B=B(x,r),B'=B(y,r)$. It is clear that $[B]^{\varepsilon }$ and $[B']^{\varepsilon }$ are grid-connected sets. Since $B\cap B'$ contains a ball with radius $r/2$, and therefore a square of side-length $2^{-1/2}r\geqslant \varepsilon $,  there is one point of $Z_{\varepsilon }$ in $B\cap B'$, which connects $[B\cup B']^{\varepsilon }=[B]^{\varepsilon }\cup [B']^{\varepsilon }$.

From there, given a sequence of points $x_{1},\dots ,x_{q}\in A$ such that $\|x_{i}-x_{i+1}\|\leqslant r$, $[\cup _{i}B(x_{i},r)]^{\varepsilon }$
is connected. Any given two points of $A$ can be connected by such a sequence, whence $[A^{\oplus r}]^{\varepsilon }$ is connected.
\end{proof}

 Assume first  that $F$ is simply connected, or equivalently 
 that $F$ and $F^{c}$ are connected.  Covering $\partial F$ with a finite number of balls inside which $\partial F$ can be represented as the graph of a $\mathcal{C}^{1}$ function, $F$ can also be represented as the bounded complement of the simple $\mathcal{C} ^{1}$ Jordan curve formed by $\partial F$.

Let us prove that $F^{\varepsilon }$ and $(F^{\varepsilon })^{c}$ are also  connected, i.e. that $F^{\varepsilon }$ is simply connected, and therefore homotopy equivalent to $F$.
Let $r<\rho $ such that $\varepsilon <2^{-1/2}r$. 

\begin{lemma} $F^{\ominus r}$ and $(F^{c})^{\ominus r}$ are connected.\end{lemma}

\begin{proof}   Since $\partial F$ has reach at least $r$, introduce the mapping 
\begin{align*}
\psi _{r}(x)=\begin{cases}
\pi _{\partial F}(x)-r\n_{F}(\pi _{\partial F}(x))$ if $x\in F,d(x,\partial F)\leqslant   r\\
x\text{ if }x\in F,d(x,\partial F)\geqslant  r.
 \end{cases}
 \end{align*}
 The definition makes sense because if $x$ is at distance exactly $r$ from $\partial F$, then the normal to $F$ in $\pi _{\partial F}(x)$ is  colinear to $x-\pi _{\partial F}(x)$ whence $x=\pi _{\partial F}(x)-r\n_{F}(x)=\psi _{r}(x)$. We also have, using Proposition \ref{prop:poperties-rollingball}, that if $d(x,\partial F)\leqslant r, d(\psi _{r}(x),\partial F)=r$.

By classical results about sets with positive reach, $\pi _{\partial F}$ is continuous on $\partial F^{\oplus r}$, see for instance Federer \cite{Fed} Th.4.8-(4).  Since $\n_{F}(\cdot )$ is Lipschitz on $\partial F$, $\psi _{r}$ is continuous on the closed set $F\cap (\partial F)^{\oplus r}$. 
Let $x\in F$. It is clear that $\psi _{r}$ is continuous in $x$ if $d(x,\partial F)\neq r$. Otherwise, we know that for $\varepsilon >0$ there is $\eta >0$ such that for $y\in F\cap B(x,\eta )\cap \partial F^{\oplus r},  | \psi _{r}(x)-\psi _{r}(y) | \leqslant \varepsilon $. It follows that $\psi _{r}$ is continuous on $F$. 
Therefore $F^{\ominus r}$ is connected as the image of the connected set $F$ under the continuous mapping $\psi _{r}$. 
Defining similarly
\begin{align*}
\psi' _{r}(x)=\begin{cases}
\pi _{\partial F}(x)+r\n_{F}(\pi _{\partial F}(x))$ if $x\in F^{c},d(x,\partial F)\leqslant   r\\
x\text{ if }x\in F^{c}\setminus (\partial F)^{\oplus r}.
 \end{cases}
 \end{align*}provides a continuous mapping from $F^{c}$ to $(\text{cl}(F^{c}))^{\ominus r}$.
 \end{proof}
 It follows by Lemma \ref{lm:homeomorphic}
 that $(F)^{\varepsilon }=((F^{\ominus r})^{\oplus r})^{\varepsilon }$ is connected, using Theorem \ref{th:blaschke}-(iii). Similarly
 $(F^{c})^{\varepsilon }=(F^{\varepsilon })^{c}$ is connected, i.e. $F^{\varepsilon }$ is simply connected.
 
 The fact that $[F]^{\varepsilon }$ is simply grid-connected and  $\Phi _{X}^{\varepsilon }([F]^{\varepsilon })=0$ (Lemma \ref{lm:PhiX}) yields that the boundary of $F^{\varepsilon }$ can also be represented by a continuous Jordan curve. 
   Recall the Brouwer-Schoenflies Theorem \cite{Cairns}, that states that given two  Jordan curves in $\mathbb{R}^{2}$, the bounded components $A,A'\subseteq \mathbb{R}^{2}$ of the respective complements of the Jordan curves   are homeomorphic. It is not hard to modify the result to see that if $A$ and $A'$ (and their complements) coincide on some open subset $B$, then the homeomorphism can be chosen to be the identity on $B$.
Therefore there is a homeomorphism $\Phi _{F}$ between $F$ and $F^{\varepsilon }$, that is the identity on $\mathbb{R}^{2}\setminus (\partial F)^{\oplus 2\varepsilon }$.

Now let   $F'$ be another bounded simply connected $\rho $-regular set such that $\partial F$ and $\partial F'$ are at distance more than $4\varepsilon $. We also have that $(F')^{\varepsilon }$ is simply connected and $\Phi _{X}([F']^{\varepsilon })=0$.  Since $\partial F\cap \partial F'=\emptyset $, we either have $F\subseteq F', F'\subseteq F$, or $F\cap F'=\emptyset $. Since $\partial ((F')^{\varepsilon })\subseteq (\partial (F'))^{\oplus 2\varepsilon }$ and $\partial (F^{\varepsilon })\subseteq (\partial F)^{\oplus 2\varepsilon }$, the sets $[F]^{\varepsilon }$ and $[F']^{\varepsilon }$ respect the same inclusion hierarchy as $F$ and $F'$ (i.e. either $F^{\varepsilon }\subseteq (F')^{\varepsilon }, F^{\varepsilon }\cap (F')^{\varepsilon }=\emptyset ,$ or $(F')^{\varepsilon }\subseteq F^{\varepsilon }$). It is also clear that $\Phi _{X}^{\varepsilon }([F\cup F'])=\Phi _{X}^{\varepsilon }([F])+\Phi _{X}^{\varepsilon }([F'])=0$ because $\partial F$ and $\partial F'$ are at distance more than $4\varepsilon $. The sets on which $\Phi _{F}$ and $\Phi _{F'}$  are not the identity are resp. contained in $(\partial F)^{\oplus 2\varepsilon },(\partial F')^{\oplus 2\varepsilon }$, whence they are disjoint. Their composition $\Phi _{F}\circ \Phi _{F'}=\Phi _{F'}\circ \Phi _{F}$ gives a homeomorphism between $F\cup F'$ and $(F\cup F')^{\varepsilon }$ if $F$ and $F'$ are disjoint, or between $F\setminus F'$ and its digitalisation $(F\setminus F')^{\varepsilon }$ if $F'\subseteq F$.

Applying this process inductively  yields that given any set  $A$ built by adding or removing recursively a finite family of $\rho $-regular connected components $C_{i}$ satisfying $d(\partial C_{i},\partial C_{j})>4\varepsilon $ for $i\neq j$, $A$ and $A^{\varepsilon }$ are homeomorphic. 

The connected components of $F$ and $F^{c}$ are indeed $\rho $-regular, and their boundaries are by hypothesis at pairwise distance $>4\varepsilon $. Studying the inclusion hierarchy of the elements of $\Gamma (F)$ and $\Gamma (F^{c})$ yields that $F$ can be built by adding or removing those components in such a way, whence the result is proved for any $\rho$-regular set $F$ with $r,\varepsilon $  chosen like above.

\end{proof}

{\bf (ii)} We now have to deal with the intersection with $W$. Drop the assumption that $\partial F\cap \partial W=\emptyset $.  
Introduce the set $E=(\partial F\cap \partial W)\cup (\corners(W)\cap F)$.
 We will ``round'' the sharpness that $F\cap W$ has in each point of $E$ to have a new set $F'\subseteq W$ that is  $\rho$-regular in the neighbourhood of $E$ and coincides with $F\cap W$ away from $E$, and prove that it does not modify the topology of $F\cap W$, or that of $[F\cap W]^{\varepsilon }$.

For $z\in \partial F\cap \partial W,$ an application of the Implicit Function Theorem around $z$ yields that for $r$ small enough, $\partial F\cap B[z,r]$ can be  represented as the graph of a univariate  $\mathcal{C}^{1}$ function $\varphi $. The Lipschitzness of $\n_{F}(\cdot )$ yields the Lipschitzness of $\varphi' $, which in turn yields that for $r,\varepsilon $ sufficiently small,     $\partial F\cap W\cap  B[z,r]$ is connected and $[F\cap W]^{\varepsilon }\cap B[z,r]\cap W\neq \emptyset $.
Fix $0<r<\rho $ such that 
\begin{itemize}
\item  For $x\in \partial F\cap \partial W$ and $r'\leqslant r,$    $\partial F\cap B[x,r']\cap W$ is connected, { and for $\varepsilon  \leqslant r'$ sufficiently small}, $[F\cap W]^{\varepsilon }\cap B[z,r']\cap W\neq \emptyset $ and $[F^{c}\cap W]^{\varepsilon }\cap B[z,r']\cap W\neq \emptyset $, see above.
\item  Points of $E$ are at pairwise $\|\cdot \|_{\infty }$-distance strictly more than $2r$.
\item For $x\in \partial F\cap \partial W,y\in B[x,r]\cap \partial F$, $\n_{F}(y)$ is not colinear with $\n_{W}(x)$.
\end{itemize}

Let $z\in \partial F\cap \partial W$. Up to doing rotations, assume that  $\n_{W}(z)=-\u_{2}$.
   Below and  in the sequel of the proof, $B:=B[z,r]$.
 We introduce a result proving that removing $F\cap W\cap \text{\rm{int}}(B)$ form $F\cap W$  does not change its \EC. Denote, for $x\in \mathbb{R}^{2}$, $L_{x}:=\{x-t\u_{1};t\geqslant 0\}$, the half-line on the left of $x$. 
 \begin{lemma} 
\label{lm:local-modif}
 Let $A\in \mathcal{A}$ be such that \begin{enumerate}
\item for $x\in A\cap B, (L_{x}\cap B)\subseteq A$
\item $A\cap \partial B$ is non-empty and connected
\item $B\setminus A\neq \emptyset .$
\end{enumerate}Then $A\setminus \text{\rm{int}}(B)\in \mathcal{A}$ and $\chi (A)=\chi (A\setminus  \text{\rm int}(B))$.
 \end{lemma}
 
 \begin{proof} 
 The first hypothesis yields 
  $A\cap B=\cup _{x\in A\cap B}L_{x}\cap B.$ Since $\partial B\cap A$ is connected, two points $x,x'$ of $B\cap A$ are connected in $A$ through $((L_{x}\cup L_{x'})\cap B)\cup (\partial B\cap A)$. 
 Removing $A\cap \text{\rm{int}}(B)$ from $A\cap B$ amounts to removing a piece of the connected component of $A$ containing $\partial B\cap A$, without splitting it (because $\partial B\cap A$ keeps it connected). Therefore, $\#\Gamma (A)=\#\Gamma (A\setminus\text{\rm{int}}(B))$.    Since $ \partial B\setminus A$ is not empty and connected, a similar representation involving the $R_{x}:=\{x+t\u_{1},t>0\},x\in B$ yields $\#\Gamma (A^{c})=\#\Gamma (A^{c}\setminus \text{int}(B)).$ We also  easily have that $A^{c}\cap \partial B$ is non-empty and simply connected, thus $\#\Gamma (A^{c}\setminus \text{\rm{int}}(B))=\#\Gamma (A^{c}\cup \text{\rm{int}}(B))$. It then yields
\begin{align*}
\#\Gamma (A^{c})=\#\Gamma (A^{c}\cup \text{\rm{int}}(B))=\#\Gamma ((A\setminus \text{\rm{int}}(B))^{c}),
\end{align*}and $\chi (A)=\chi (A\setminus \text{\rm{int}}(B))$. 
  \end{proof}

We apply this Lemma to $  A=F\cap W$.  The definition of $r$ yields that $\partial F\cap B$ is connected and that   $n_{F}(\cdot )_{[1]}  >0$ on  $ \partial F\cap B$. 
Let us prove that for  $x\in A, (L_{x}\cap B)\subseteq A$. The fact that $(L_{x}\cap B)\subseteq W$ stems from the fact that the closest corners of $W$ are at $\|\cdot \|_{\infty }$-distance at least $2r$ by hypothesis on $r$.
 Let $s=\inf\{t>0:x-t\u_{1}\notin F\}$. 
 Put $y=(x_{[1]},x_{[2]}-s)$, $y$ is in $ \partial F$. Since the part of $L_{x}$ immediately on the right of $y$ is in $F$, the outside rolling ball of $F$ in $y$  contains a portion of $L_{x}$ on the left of $y$, whence  the unit vector between $y$ and the centre of this ball, $\n_{F}(y)$, satisfies $\n_{F}(y)_{[1]}<0$. It follows that $y\notin B$, whence $(L_{x}\cap B)\subseteq F$. This property, combined with the connectedness of $\partial F\cap B$, entails that $\partial F\cap B\cap W$ is connected, whence Lemma \ref{lm:local-modif} applies. The same reasoning still holds if the value of $r$ is decreased.

Call 
\begin{align*}
C_{r}=\bigcup _{z\in E}B[z,r].
\end{align*}  Repeat the  procedure above for every point of $E$. 
Even though points of $\text{\rm{corners}}(W)\cap F$ are not points of $\partial F\cap \partial W$, they can geometrically be treated in the same way.
As the $B[z,r],z\in E$, don't overlap,  $\chi (F\cap W)=\chi( (F\cap W)\setminus \text{int}(C_{r}))$, and this equality holds for smaller values of $r$.

The idea is now to replace  $F\cap W$  by a smooth connected set $F'$ that coincides with $F\cap W$ outside $B$, and do the same around every other point of $E$. Let   $0<r' <r/5$.  Introduce
\begin{align*}
F'=((F\cap  W)^{\ominus r'})^{\oplus r'},
\end{align*} and put $B'=B[z,3r']$. 
Recall that by Theorem \ref{th:blaschke}-(iii), since $r'<\rho $, $F=(F^{\ominus  r'})^{\oplus r'}$.
Let $x\in W\setminus \partial W^{\oplus2 r'}.$ Note that every $y$ such that $x\in B(y,r')$ satisfies $B(y,r')\subseteq W$. Then,
\begin{align*}
x\in F'&\Leftrightarrow x\in B(y,r')\text{ for some }y\in F\cap W\text{ such that }B(y,r')\subseteq F\cap W\\
&\Leftrightarrow x\in B(y,r')\text{ for some }y\in F\text{ such that }B(y,r')\subseteq F\\
&\Leftrightarrow x\in (F^{\ominus r'})^{\oplus r'}=F,\\
&\Leftrightarrow x\in F\cap W.
\end{align*}
 This yields
\begin{align}
\label{eq:coincide-Fprime-1}
F'\setminus \partial W^{\oplus 2r'}=(F\cap W)\setminus \partial W^{\oplus 2r'}.
\end{align}

For similar reasons, if $x\in W$ is at distance more than $2r'$ from $\partial F$ { or from a corner of $W$}, then $x\in F'\Leftrightarrow x\in (F\cap W)$.  This yields
\begin{align}
\label{eq:coincide-Fprime-2}
F'\setminus (\partial F\cup \corners(W))^{\oplus2 r'}=(F\cap W)\setminus (\partial F\cup \text{corners}(W))^{\oplus2 r'}.
\end{align}
Intersecting 
\eqref{eq:coincide-Fprime-1}
and 
\eqref{eq:coincide-Fprime-2}
yields $(F\cap W)\setminus C_{2r'}=F'\setminus C_{2r'}$.    It entails in particular $\partial B'\cap F'= \partial B'\cap (F\cap W)$, and it is a connected set because $r'\leqslant r$.

 Therefore, to apply Lemma \ref{lm:local-modif} to $F'$ and $B'$, it remains to show that for $x\in F'\cap B'$, $(L_{x}\cap  B')\subseteq F'$.
A point $x$ of $B'$ is in $F'$ if  there is a ball $B(y,r')\subseteq (F\cap W)$ such that $x\in B(y,r')$. Since $x\in B'$ and ${5r'<r},$ $B(y,r')\subseteq  (F\cap W)\cap B$. We already proved that for all point $w$ of $B(y,r')$, since $w\in F\cap W\cap B$, $(L_{w}\cap B)\subseteq F\cap W$, whence in particular, for $t\geqslant 0$, $(B(y-t\u,r')\cap B)\subseteq (F\cap W\cap B)$. Each point $w=x-t\u_{1}$ of $L_{x}\cap B'$ is  in $B(y-t\u,r')$, whence it is in $F'$. Therefore, Lemma \ref{lm:local-modif} applies. Applying the same reasoning in each point of $E$ entails 
\begin{align*}
\chi (F')=\chi (F'\setminus \text{\rm{int}}(C_{2r'}))=\chi ((F\cap W)\setminus \text{\rm{int}}(C_{2r'}))=\chi (F\cap W).
\end{align*}

Since $F'$ satisfies Theorem \ref{th:blaschke}-(iii), $F'$ is $\rho$-regular.
{  Let us decrease the value of  $ \varepsilon$ so that $2\varepsilon <r'$ and  $(F')^{\varepsilon }$ is homeomorphic to $F'$ (see Lemma \ref{lm:reg-set-approx})}. In particular, $\chi ((F')^{\varepsilon })=\chi (F')=\chi (F\cap W)$. The end of the proof consists in showing $\chi ((F')^{\varepsilon })=\chi ((F\cap W)^{\varepsilon })$, and computing the latter quantity.
 
Using the fact that Lemma \ref{lm:local-modif} applies to $F'$ and $B'$, we can easily prove that given a point $x\in (F')^{\varepsilon }\cap B'$, all the points $x-t\u_{1},t>0$ that are in $B'$ are also in $(F')^{\varepsilon }$. 
Recall that $r,\varepsilon $ have been chosen so that $[F\cap W]^{\varepsilon }\cap B[z,s]\neq \emptyset ,[F^{c}\cap W]\cap B[z,s]\neq \emptyset $ for $s\leqslant r$. 
 Therefore, applying Lemma \ref{lm:local-modif} gives $\chi ((F')^{\varepsilon })=\chi ((F')^{\varepsilon }\setminus C_{2r'})$. 
A similar reasoning gives that $\chi ((F\cap W)^{\varepsilon })=\chi ((F\cap W)^{\varepsilon }\setminus \text{\rm{int}}(C_{2r'}))=\chi ((F\cap W)^{\varepsilon }\setminus \text{\rm{int}}(C_{2r'}))$.

Since $F'$ and $F\cap W$ coincide outside of $ B[z,2r']$, only pixels centred in $B[z,2r']$ can differentiate $[F']^{\varepsilon }$ and $[F\cap W]^{\varepsilon }$. 
Since $2r'+2\varepsilon <3r'$, a pixel centred in $B[z,2r']$ cannot reach $\partial B'$, whence
\begin{align*}
(F')^{\varepsilon }\setminus \text{\rm{ int }}(B')=(F\cap W)^{\varepsilon }\setminus \text{\rm{ int }}(B'),
\end{align*}
and applying a similar reasoning in every point of $E$ yields that $\chi (F\cap W)=\chi (F')=\chi  ((F')^{\varepsilon })=\chi ((F\cap W)^{\varepsilon })$. Remark that, as announced, this equality holds for $\varepsilon $ smaller than some value $\varepsilon (F,W)$ that is invariant under simultaneous translations of $F$ and $W$.

\textsl{  Proof of \eqref{eq:euler-bicovariograms}}.
 Let $\varepsilon <\varepsilon (F,W)=\varepsilon (F+\varepsilon y,W+\varepsilon y),y\in \mathbb{R}^{2}$. We proved that $\Phi ^{\varepsilon }_{X}([F\cap W]^{\varepsilon })=0$, whence $\chi (F\cap W)^{\varepsilon })=\chi ^{\varepsilon }([F\cap W]^{\varepsilon })$. Using \eqref{eq:discrete-EC} yields \eqref{eq:euler-pixelized}, and
    \begin{align*}
\chi (F\cap W)&=\int_{[0,1)^{2}}\chi ((F\cap W)+\varepsilon y)dy=\int_{[0,1)^{2}}\chi ^{\varepsilon }([(F\cap W)+\varepsilon y]^{\varepsilon })dy\\
&=\sum_{x\in \varepsilon \mathbb{Z} ^{2}} \int_{[0,1)^{2}}\chi ^{\varepsilon } (x,F\cap W+\varepsilon y)dy\\
&=\sum_{x\in  \mathbb{Z} ^{2}} \int_{[0,1)^{2}}\chi ^{\varepsilon }(\varepsilon x,F\cap W+\varepsilon y)dy\\
&=\sum_{x\in \mathbb{Z} ^{2}} \int_{[0,1)^{2}}\chi ^{\varepsilon }(\varepsilon (x-y),F\cap W)dy=\int_{\mathbb{R}^{2}}\chi ^{\varepsilon }(\varepsilon y,F\cap W)dy\\
&=\varepsilon ^{-2}\int_{\mathbb{R}^{2}}\chi ^{\varepsilon }(y,F\cap W)dy=\varepsilon ^{-2}(\delta_0 ^{\varepsilon \u_{1},\varepsilon \u_{2}}-\gamma _{-\varepsilon \u_{1},-\varepsilon \u_{2}}^{0})(F\cap W).
\end{align*}

  \subsection{ Proof of  Theorem \ref{th:bound-pixel-comp}}
  \label{sec:proof-entang}
  
 We call $D_{1},\dots ,D_{m}$ the connected components of $(F\cap W)\cup (F\cap W)^{\varepsilon }$, and for $1\leqslant i\leqslant m$, call $C_{i,j}^{\varepsilon },1\leqslant j\leqslant k_{i}$, the connected components of $(F\cap W)^{\varepsilon }$ hitting $D_{i}$. Since a connected component of $(F\cap W)^{\varepsilon }$ cannot hit   $2$ distinct $D_{i}$'s, we have $\#\Gamma ((F\cap W)^{\varepsilon })=\sum_{i=1}^{m}k_{i}$.
Since a connected component of $F\cap W$ cannot hit two distinct $D_{i}$'s either,  
 $
m\leqslant \#\Gamma (F\cap W).
 $
We yet have to control the number of components of $(F\cap W)^{\varepsilon }$ that hit a given $D_{i}$. Define  $C_{i,j}=[C_{i,j}^{\varepsilon }] $ (not necessarily connected in $\varepsilon \mathbb{Z} ^{2}$ as pixels of $C_{i,j}^{\varepsilon }$ might only touch through a corner, but it has no consequences). 
 
 The plan of the rest of the proof is the following: we are going to associate to each $C_{i,j},1\leqslant i\leqslant m,j>1,$ either a corner of $W$, an element of $\N_{\varepsilon }'(F,W)$, or an element of $\N_{\varepsilon} (F)\cap W^{\oplus \varepsilon }$, in such a way that this element cannot be associated to more than two distinct  $C_{i,j}$. This will indeed yield the bound \eqref{eq:bound-Gamma-eps-W}.

Here again, the boundary effects of the intersection with $W$ are cumbersome for the proof, and require to introduce some notation. Recall that $\partial ^{\varepsilon }[W]$ stands for the set formed by the elements of $[W]^{\varepsilon } $ that have at least one grid neighbour outside $W$.  
Assume that some $C_{i,j},j\geqslant1,$ touches $\partial ^{\varepsilon }[W]$. Call $\emph{root}$ of $C_{i,j}$ (or $C_{i,j}^{\varepsilon })$  a   horizontal or vertical   pixel interval $\llbracket x,y\rrbracket$ such that \begin{itemize}
\item $\llbracket x,y\rrbracket \subseteq   \partial ^{\varepsilon }[W]$
\item  $\llbracket x,y\rrbracket \cap  C_{i,j}\neq \emptyset $
\item every $z\in \llbracket x,y\rrbracket $ has at least a grid neighbour in $F$
\item for $(i',j')\neq (i,j)$, $\llbracket x,y\rrbracket \cap C_{i',j'}= \emptyset $
\item it is maximal, in the sense that no other pixel interval satisfying these properties strictly contains $\llbracket x,y\rrbracket$. 
\end{itemize}

 Call $\root{C_{i,j}}\subseteq  Z_\varepsilon  $ the union of $C_{i,j}$ with its roots, and $\root{C_{i,j}^{\varepsilon }}$ the  union of pixels centred in $\root{C_{i,j}}$ . See Figure \ref{fig:legs} for examples. 
 Fix $1\leqslant i\leqslant m,$ such that $k_{i}>1$, and let $1<j\leqslant k_{i}$. We distinguish between three cases.

 \begin{figure} 
 \centering
  \caption{  \label{fig:legs}A set where two disjoint components $C_{i,j}$ have roots. One component has two roots and another has only one. The pixels centred in roots are coloured in grey,  root overlap is represented in darker grey.}
   \includegraphics[scale=.2]{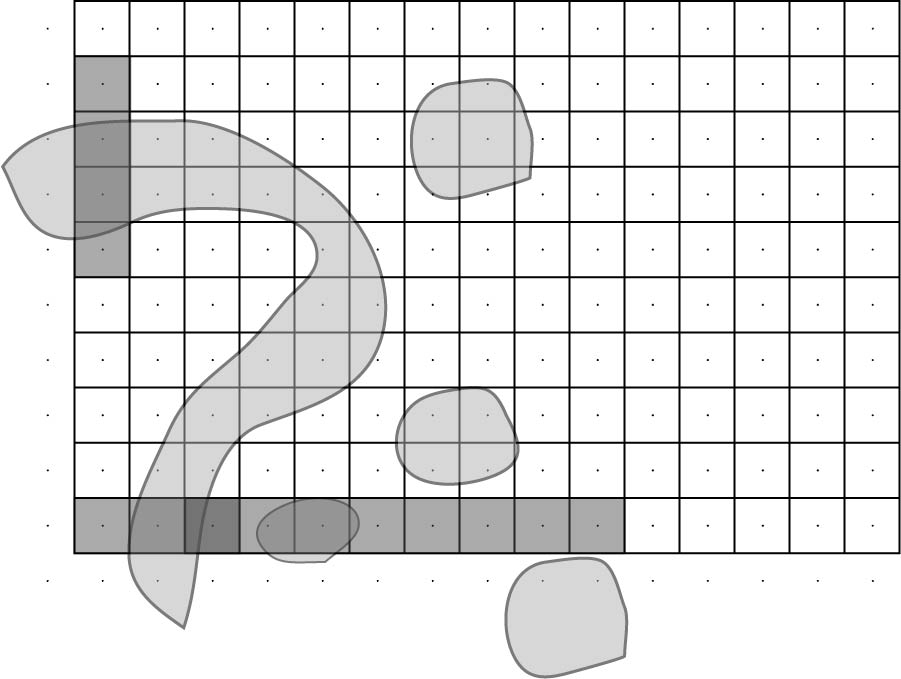} 
 \end{figure}

{\textbf{(i)}}  If   a root of $C_{i,j}$  contains a corner at one of its extremities (as is the case for instance for a root of the big component in Figure \ref{fig:legs}), associate this corner to $C_{i,j}^{\varepsilon }$, and remark that a corner cannot be associated to more than two connected components of  $(F\cap W)^{\varepsilon }$, using the fact that a root of a $C_{i,j}$ cannot contain a root of a $C_{i',j'}$, $(i',j')\neq (i,j)$.

\textbf{ (ii)} Assume that  no corner can be associated to $C^{\varepsilon }_{i,j}$, i.e. that none of its roots touches a corner, but that  for some $(i',j')\neq (i,j), $ $\root{C_{i,j}}\cap \root{C_{i',j'}}\neq \emptyset $. Let  $I=\llparenthesis x,y\rrparenthesis\subseteq  (\root{C_{i,j}}\cap \root{C_{i',j'}}) $ be a maximal  grid interval of this intersection, with $x,y\in  Z_\varepsilon $.  Since any root of $C_{i,j}$ does not touch $C_{i',j'}$, and vice-versa, $I\subseteq [F^{c}]$. There is an example of such a root intersection in Figure \ref{fig:legs}.
 
 One extremity of $I$, say $x$, is in $C_{i,j}$, and the other, $y$, is in $C_{i',j'}$.
 Since $\llparenthesis x,y\rrparenthesis \subseteq (\root{C_{i,j}}\setminus C_{i,j}),$ the definition of a root entails that all the points of $\llparenthesis x,y\rrparenthesis$ are in $[F^{c}]$ and are at distance $\leqslant  \varepsilon $ from $F$. Therefore $\{x,y\}\in \N_{\varepsilon }'(F,W)$. Also, $\{x,y\}$ can only be associated to $C_{i,j}$ and $C_{i',j'}$ in this manner. 

\textbf{(iii)} Assume that $C_{i,j}$ does not fit in {\bf (i)} or in {\bf (ii)}. 
 Recall that we assumed  $k_{i}>1$ and $1<j\leqslant  k_{i}$. There is a continuous path $\gamma:[0,1]\to D_{i} $ going from $ F^{\varepsilon }\setminus C_{i,j}^{\varepsilon }$ to   $C_{i,j}^{\varepsilon }$ with $\gamma ((0,1))\subseteq F\setminus F^{\varepsilon }$. In particular, $\gamma((0,1)) $ avoids $[F]$. 
Discarding the two first cases means that somehow $\gamma $ cannot arrive to $C_{i,j}^{\varepsilon }$ from another component $C_{i,j'}^{\varepsilon }, j'\neq j,$ by creeping along an edge of $W$ adjacent to $C_{i,j}$, if there is such. More precisely, There is no $(i,j')\neq (i,j)$ such that $\root{C_{i,j}}\cap  \root{C_{i,j'}}\neq \emptyset $.

Now, say that $x,y\in   Z_\varepsilon $ form a boundary pair for $\root{C_{i,j}}$, denoted $\{x,y\}\in \Delta  \root{C_{i,j}}$, if the square $\p_{x,y}$ has exactly one edge touching $\root{C_{i,j}^{\varepsilon }}$. Remark that only one of the two edges not containing $x$ or $y$ can fulfill this condition, and  the corresponding component of $\p'_{x,y}$, denoted $\p_{x,y}^{in}$, touches $\root{C_{i,j}^\varepsilon} $,  while the other, denoted $\p_{x,y}^{\text{out}}$, does not.

The point $\gamma (0)$ is on the boundary of a pixel that is not contained in $\root{C_{i,j}^{\varepsilon} },$ but $\gamma (\cdot )$ eventually enters in such a pixel. Therefore, studying possible local configurations of $\root{C_{i,j}^{\varepsilon }}$'s boundary entails that there is $t>0$ such that $\gamma (t)\in \p_{x,y}^{\text{out}}$ for some $\{x,y\}\in \Delta  \root{C_{i,j}}$. Let $t^{\text{out}}$ be the latest such time.  The rest of the proof consists in showing $\{x,y\}\in \N_{\varepsilon }(F)$.

Since $\root{C_{i,j}^{\varepsilon} }$ is only in contact with such $\p_{x,y}$ through $\p_{x,y}^{in}$, $\gamma (t)$ will eventually have to get out of the pixel $\p_{x,y}$ at some time $t^{in}$, without touching $\p_{x,y}^{\text{out}}$ because the latest such time has already been reached, therefore $\gamma (t^{in})\in \p_{x,y}^{in}$, and $\gamma ([t^{\text{out}},t^{in}])\subseteq \p_{x,y}$ (also note that $\gamma $ cannot pass through $x$ or $y$ because $\gamma ((0,1))\subseteq (F\setminus Z_{\varepsilon }$)).
Therefore $\tilde \gamma \subseteq F$ connects the two components of $\p'_{x,y}$ through $\p_{x,y}$. This is the primary condition for $\{x,y\}\in \N_{\varepsilon }(F)$.

 To complete the proof that $\{x,y\}\in \N_{\varepsilon }(F)\cap W^{\oplus \varepsilon }$, it remains to show that $x,y\in [F^{c}]\cap W^{\oplus \varepsilon }$. Since $x,y\in \tilde\gamma ^{\oplus \varepsilon }$ and $\tilde\gamma \subseteq W$, $x,y\in W^{\oplus \varepsilon }$.
None of them is in $[F\cap W]$ otherwise it would be in a different component $C_{i',j'}^{\varepsilon }$ that would touch $C_{i,j}^{\varepsilon }$ (eventually through a corner), possibility that has been treated in \textbf{(i)} or {\bf (ii)}. They cannot be both outside $W$ because $\tilde\gamma \cap [x,y]\neq \emptyset $ and $\tilde\gamma \subseteq W$. The two last possibilities are that $x,y\in [F^{c}]$, or $y\in [W\cap F^{c}] $ and $x\in [F\cap W^{c}]$ (or the other way around). The second possibility is also discarded because it entails $y\in \root{C_{i,j}}$, which contradicts $\{x,y\}\in \Delta C_{i,j}$ (also using the fact that no corner lies in a root of $C_{i,j}^{\varepsilon }$).

We associate this pair $\{x,y\}\in \N_{\varepsilon }(F)\cap W^{\oplus \varepsilon }$ to $C_{i,j}^{\varepsilon }$ and note that since  $x$ and $y$ lie at $\|\cdot \|_{\infty }$-distance $\varepsilon/2 $ from $\root{C_{i,j}^{\varepsilon }}$, $\{x,y\}$ cannot be associated to more than two components in this way, and this concludes the proof of \eqref{eq:bound-Gamma-eps-W}.

The first inequality can easily be derived  in the same way in the case $W=\mathbb{R}^{2}$; noting that in this case $W$ has no edge and no corners. \\
}

{\bf Acknowledgement} The author is grateful to Bruno Galerne, who helped trigger the ideas contained in this paper, and to Anne Estrade, for many discussions around the topic of the \EC.

\vspace{\baselineskip}

 \bibliographystyle{plain}
    \bibliography{bicovariograms}

\begin{thebibliography}{10}

\bibitem{ABBSW}
R.~J. Adler, O.~Bobrowski, M.~S. Borman, E.~Subag, and S.~Weinberger.
\newblock Persistent homology for random fields and complexes.
\newblock {\em IMS Coll.}, 6:124--143, 2010.

\bibitem{AdlSam}
R.~J. Adler and G.~Samorodnitsky.
\newblock Climbing down {G}aussian peaks.
\newblock arXiv:1501.07151, 2015.

\bibitem{AST}
R.~J. Adler, G.~Samorodnitsky, and J.~E. Taylor.
\newblock High level excursion set geometry for non-gaussian infinitely
  divisible random fields.
\newblock {\em Ann. Prob.}, 41(1):134--169, 2013.

\bibitem{AdlTay07}
R.~J. Adler and J.~E. Taylor.
\newblock {\em Random Fields and Geometry}.
\newblock Springer, 2007.

\bibitem{AFP}
L.~Ambrosio, N.~Fusco, and D.~Pallara.
\newblock {\em Functions of Bounded Variations and Free Discontinuity
  Problems}.
\newblock Oxford University Press, 2000.

\bibitem{AKPM}
C.~H. Arns, M.~A. Knackstedt, W.~V. Pinczewski, and K.~R. Mecke.
\newblock {E}uler-{P}oincar{\'e} characteristics of classes of disordered
  media.
\newblock {\em Phys. Rev. E}, 63:031112, 2001.

\bibitem{AMMS}
C.~H. Arns, J.~Mecke, K.~Mecke, and D.~Stoyan.
\newblock Second-order analysis by variograms for curvature measures of
  two-phase structures.
\newblock {\em The European Physical Journal B}, 47:397--409, 2005.

\bibitem{AufBen13}
A.~Auffinger and G.~Ben Arous.
\newblock Complexity of random smooth functions on the high-dimensional sphere.
\newblock {\em Ann. Prob.}, 41(6):4214--4247, 2013.

\bibitem{AveBia09}
G.~Averkov and G.~Bianchi.
\newblock Confirmation of {M}atheron's conjecture on the covariogram of a
  planar convex body.
\newblock {\em Journal of the European Mathematical Society}, 11:1187--1202,
  2009.

\bibitem{AzaWsc}
J.~Aza\"is and M.~Wschebor.
\newblock A general expression for the distribution of the maximum of a
  {G}aussian field and the approximation of the tail.
\newblock {\em Stoc. Proc. Appl.}, 118(7):1190--1218, 2008.

\bibitem{BieDesEuler}
H.~Bierm\'e and A.~Desolneux.
\newblock Level total curvature integral: Euler characteristic and 2d random
  fields.
\newblock preprint HAL, No. 01370902, 2016.

\bibitem{BieDes}
H.~Bierm\'e and A.~Desolneux.
\newblock On the perimeter of excursion sets of shot noise random fields.
\newblock {\em Ann. Prob.}, 44(1):521--543, 2016.

\bibitem{BDJR}
P.~Blanchard, C.~Dobrovolny, D.~Gandolfo, and J.~Ruiz.
\newblock On the mean {E}uler characteristic and mean {B}etti's numbers of the
  {I}sing model with arbitrary spin.
\newblock {\em J. Stat. Mech.}, 2006(03):3--11, 2006.

\bibitem{Cairns}
S.~C. Cairns.
\newblock An elementary proof of the {J}ordan-{S}choenflies {t}heorem.
\newblock {\em Proc. AMS}, 2(6):860--867, 1951.

\bibitem{Mar15}
V.~Cammarota, D.~Marinucci, and I.~Wigman.
\newblock Fluctuations of the {E}uler-{P}oincar{\'e} characteristic for random
  spherical harmonics.
\newblock {\em Proc. AMS}, 144:4759--4775, 2016.

\bibitem{EstLeo14}
A.~Estrade and J.~R. Leon.
\newblock A central limit theorem for the {E}uler characteristic of a
  {G}aussian excursion set.
\newblock {\em Ann. Prob.}, 44(6):3849--3878, 2016.
\newblock Ann. Prob., to appear.

\bibitem{Fed}
H.~Federer.
\newblock Curvature measures.
\newblock {\em Trans. AMS}, 93(3):418--491, 1959.

\bibitem{Galerne}
B.~Galerne.
\newblock Computation of the perimeter of measurable sets via their
  covariogram. {A}pplications to random sets.
\newblock {\em Image Anal. Stereol.}, 30(1):39--51, 2011.

\bibitem{Hil02}
R.~Hilfer.
\newblock Review on scale dependent characterization of the microstructure of
  porous media.
\newblock {\em Transport in Porous Media}, 46(2-3):373--390, 2002.

\bibitem{HSN}
J.~Hiriart-Urrurty, J.~Strodiot, and V.~H. Nguyen.
\newblock Generalized {H}essian matrix and second-order optimality conditions
  for problems with ${C}^{1,1}$ data.
\newblock {\em Appl. Math. Optim.}, 11:43--56, 1984.

\bibitem{HLW}
D.~Hug, G.~Last, and W.~Weil.
\newblock A local {S}teiner-type formula for general closed sets and
  applications.
\newblock {\em Math. Nachr.}, 246:237--272, 2004.

\bibitem{Kid06}
M.~Kiderlen.
\newblock Estimating the {E}uler characteristic of a planar set from a digital
  image.
\newblock {\em Journal of Visual Communication and Image Representation},
  17(6):1237--1255, 2006.

\bibitem{KilFri}
J.~M. Kilner and K.~J. Friston.
\newblock Topological inference for {EEG} and {MEG}.
\newblock {\em Ann. Appl. Stat.}, 4(3):1272--1290, 2010.

\bibitem{KSS}
S.~Klenk, V.~Schmidt, and E.~Spodarev.
\newblock A new algorithmic approach to the computation of {M}inkowski
  functionals of polyconvex sets.
\newblock {\em Computational Geometry}, 34(3):127--148, 2006.

\bibitem{LacEC2}
R.~Lachi\`eze-Rey.
\newblock Euler characteristic of random fields excursions.
\newblock preprint.

\bibitem{Lan02}
C.~Lantu{\'e}joul.
\newblock {\em Geostatistical Simulation: Models and Algorithms}.
\newblock Springer, Berlin, 2002.

\bibitem{Mel90}
A.~L. Melott.
\newblock The topology of large-scale structure in the universe.
\newblock {\em Physics Reports}, 193(1):1 -- 39, 1990.

\bibitem{Mol05}
I.~Molchanov.
\newblock {\em Theory of random sets}.
\newblock Springer-Verlag, London, 2005.

\bibitem{NOP}
W.~Nagel, J.~Ohser, and K.~Pischang.
\newblock An integral-geometric approach for the {E}uler-{P}oincar\'e
  characteristic of spatial images.
\newblock {\em Journal of Microscopy}, 198:54--62, 2000.

\bibitem{Okun}
B.~L. Okun.
\newblock Euler characteristic in percolation theory.
\newblock {\em J. Stat. Phys.}, 59(1-2):523--527, 1990.

\bibitem{Schmalzing}
J.~Schmalzing, T.~Buchert, A.~L. Melott, V.~Sahni, B.~S. Sathyaprakash, and
  S.~F. Shandarin.
\newblock Disentangling the cosmic web. {I}. morphology of isodensity contours.
\newblock {\em The Astrophysical Journal}, 526(2):568, 1999.

\bibitem{SchWei}
R.~Schneider and W.~Weil.
\newblock {\em Stochastic and {I}ntegral {G}eometry}.
\newblock Probability and its Applications. Springer-Verlag, Berlin, 2008.

\bibitem{SWRS}
C.~Scholz, F.~Wirner, J.~G\"otz, U.~R\"ude, G.E. Schr\"oder-Turk, K.~Mecke, and
  C.~Bechinger.
\newblock Permeability of porous materials determined from the {E}uler
  characteristic.
\newblock {\em Phys. Rev. Lett.}, 109(5), 2012.

\bibitem{Serra}
J.~Serra.
\newblock {\em Image Analysis and Mathematical Morphology}.
\newblock Academic Press, London, 1982.

\bibitem{Sva14}
A.~Svane.
\newblock Local digital estimators of intrinsic volumes for boolean models and
  in the design-based setting.
\newblock {\em Adv. Appl. Prob.}, 46(1):35--58, 2014.

\bibitem{Sva15}
A.~Svane.
\newblock Local digital algorithms for estimating the integrated mean curvature
  of r-regular sets.
\newblock {\em Disc. Comp. Geom.}, 54(2):316--338, 2015.

\bibitem{TayWor}
J.~E. Taylor and K.~J. Worsley.
\newblock Random fields of multivariate test statistics, with applications to
  shape analysis.
\newblock {\em Ann. Stat.}, 36(1):1--27, 2008.

\bibitem{Walther}
G.~Walther.
\newblock On a generalization of {B}laschke's rolling theorem and the smoothing
  of surfaces.
\newblock {\em Math. Meth. Appl. Sci.}, 22:301--316, 1999.

\end{thebibliography}
    
 \end{document}